\documentclass[3p,12pt]{elsarticle}% {amsart}
\usepackage{amsmath}
\usepackage{amsfonts}
\usepackage{geometry}                % See geometry.pdf to learn the layout options. There are lots.
\usepackage[parfill]{parskip}    % Activate to begin paragraphs with an empty line rather than an indent
\usepackage{graphicx}
\usepackage{amssymb}
\usepackage{epstopdf}
\usepackage{psfrag}
\newcommand{\figref}[1]{{Figure~\ref{#1}}}
\newcommand {\be}{\begin{equation}}
\newcommand {\bes}{\begin{displaymath}}
\newcommand {\es}{\end{displaymath}}
\newcommand {\e}{\end{equation}}

\newcommand {\EE}{\mathbb{E}}

\newcommand {\bea}{\begin{eqnarray}}
\newcommand {\ea}{\end{eqnarray}}
\newtheorem{proposition}{Proposition}[section]
\newtheorem{theorem}{Theorem}[section]
\newtheorem{Assumption}{Assumption}[section]
\newtheorem{lemma}{Lemma}[section]
\newtheorem{remark}{Remark}[section]

\newtheorem{corollary}{Corollary}[section]
\newenvironment{proof}[1][Proof]{\textbf{#1.} }{\hspace{\stretch{1}}\rule{0.5em}{0.5em}}
\newcommand{\HH}{\mathbb{H}}
\usepackage{xspace}

\newcommand{\Dt}{\Delta t}
\usepackage{subfigure}

\newcommand{\thmref}[1]{{Theorem~\ref{#1}}}
\newcommand{\lemref}[1]{{Lemma~\ref{#1}}}
\newcommand{\secref}[1]{{Section~\ref{#1}}}
\newcommand{\assref}[1]{{Assumption~\ref{#1}}}
\newcommand{\propref}[1]{{Proposition~\ref{#1}}}
\newcommand{\coref}[1]{{Corollary~\ref{#1}}}
\newcommand{\rmref}[1]{{Remark~\ref{#1}}}
\journal{Applied Numerical Mathematics}

\begin{document}
\begin{frontmatter}
\title{Weak convergence for a stochastic exponential integrator and finite element discretization of SPDE for multiplicative  \& additive noise}
\author[at,atb]{Antoine Tambue \corref{cor1}}
\cortext[cor1]{Corresponding author}
\ead{antonio@aims.ac.za}
\address[at]{The African Institute for Mathematical Sciences(AIMS) and Stellenbosh University,
6-8 Melrose Road, Muizenberg 7945, South Africa}
\address[atb]{Center for Research in Computational and Applied Mechanics (CERECAM), and Department of Mathematics and Applied Mathematics, University of Cape Town, 7701 Rondebosch, South Africa.}
\author[jmt]{Jean Medard T Ngnotchouye}
\ead{Ngnotchouye@ukzn.ac.za}
\address[jmt]{School of Mathematics, Statistics and Computer Science,\\
University of KwaZulu-Natal, Private Bag X01, Scottsville 3209, Pietermaritzburg,
South Africa}

\begin{abstract}
We consider  a finite element approximation of a general  semi-linear stochastic partial differential equation
 (SPDE) driven by space-time multiplicative and additive noise.  
We examine the  full weak convergence  rate of the exponential Euler scheme when  the  linear operator is self adjoint  and provide preliminaries results toward  
the full weak convergence rate for non-self-adjoint linear operator.  Key part of the proof does not rely on Malliavin calculus. 
Depending of  the regularity of  the noise and  the initial solution, we found that  in some cases the rate of weak convergence is  twice the rate of the strong  convergence.
Our convergence rate  is in agreement with some numerical results in two dimensions.
%The method can be used to solve other equations, exhibiting singular shocks.

\end{abstract}
\end{frontmatter}

\section{Introduction}
The weak numerical approximation of an It\^{o}
stochastic partial differential equation  defined in $\Omega\subset
\mathbb{R}^{d}$ is analyzed. Boundary conditions on the domain $\Omega$
are typically Neumann, Dirichlet or Robin conditions.
More precisely, we consider in  the abstract setting the following stochastic partial differential equation
\begin{eqnarray}
  \label{adr}
  dX=(AX +F(X))dt + B(X)d W, \qquad  X(0)=X_{0},\qquad t \in [0, T],\qquad T>0
\end{eqnarray}
on the Hilbert space $L^{2}(\Omega)$. Here the linear operator 
$A$ which  is not necessarily selfadjoint, is the generator of an analytic semigroup $S(t):=e^{t A}, t\geq 0.$
 The functions $F$ and $B$  are nonlinear functions of
$X$ and the noise term $W(t)$ is a $Q$-Wiener process  defined on a filtered
 probability space $(\mathbb{D},\mathcal{F},\mathbb{P},\left\lbrace F_{t}\right\rbrace_{t\geq 0})$,
that is white in time. The filtration is assumed to fulfill the usual conditions 
(see e.g. \cite[Definition 2.1.11]{PrvtRcknr}). For technical reasons more interest will be on a deterministic initial value $X_0\in H$.
The noise can be represented as a series in the eigenfunctions of the covariance
operator $Q$  given by
\begin{eqnarray}
  \label{eq:W}
  W(x,t)=\underset{i \in
    \mathbb{N}^{d}}{\sum}\sqrt{q_{i}}e_{i}(x)\beta_{i}(t),
\end{eqnarray}
where $(q_i,e_{i})$, $i\in \mathbb{N}^{d}$ are the eigenvalues and  eigenfunctions
of the covariance operator $Q$ and $\beta_{i}$ are independent and identically distributed
standard Brownian motions.
Under some technical assumptions
it is well known (see  \cite{DaPZ,PrvtRcknr,Chw})
that the unique mild
solution of \eqref{adr} is given by
\begin{eqnarray}
\label{eq1}
X(t)=S(t)X_{0}+\displaystyle\int_{0}^{t}S(t-s)F(X(s))ds +\displaystyle\int_{0}^{t}S(t-s)B(X(s))dW(s).
\end{eqnarray}
Equations of type \eqref{adr} arise  in physics, biology and engineering \cite{shardlow05, AtThesis,SebaGatam} and in few cases, exact solutions exist.
The study of numerical solutions of SPDEs is therefore an active
research area and there is an extensive literature on numerical methods for SPDE \eqref{adr} \cite{Jentzen2,Jentzen3,Jentzen4,AtThesis,GTambueexpo,allen98:_finit, Stig1, Yn:04}.
Basicaly there are two types of convergence. The  strong convergence  or pathwise convergence studies  the pathwise convergence of the numerical solution to
true solution  while the  weak  convergence aims to  approximate  the law of the solution at
a fixed time. In many applications, weak error is more relevant as interest are usualy based  on some functions of the solution i.e, $\mathbb{E} \Phi (X)$, where 
$ \Phi : H \rightarrow \mathbb{R}$ and $\mathbb{E}$ is the expectation.
Strong convergence rates for  numerical approximations of stochastic evolution equations  of type \eqref{adr} with smooth
and regular nonlinearities are well understood in the scientific literature (see \cite{Jentzen2,Jentzen3,Jentzen4,allen98:_finit, Stig1, Yn:04,AtThesis,GTambueexpoM} and references therein).
Weak convergence rates for numerical approximations of equation  \eqref{adr} are far away from being well understood.  For a linear SPDE with additive noise,  the solution can be written
 explicitly  and the weak error have  been  estimated in \cite{shardlow:2003,stigweak} with implicit Euler method for time discretization. 
 The  space discretization have been performed  with finite difference method \cite{shardlow:2003, stigweak} and finite element  method  \cite{stigweak}. 
 The weak error  of the implicit Euler method  is more  complicated for  nonlinear equation of  type \eqref{adr}  as the Malliavin calculus is usually used to handle
the irregular term and the term involving the nonlinear operators  $F$ and $B$ (see \cite{Debussche,WangGan2013,stignonlinearweak}). 
In almost  all the literature for weak error estimation, the linear operator $A$ is assumed to be self adjoint. 
Furthermore no  numerical simulations were made to sustain  the theoretical results to the best of our knowledge. 
In  this paper we consider a stochastic exponential scheme (called stochastic exponential Euler scheme) as in \cite{GTambueexpoM} and provide the  weak error of the full discrete scheme (\thmref{fullweak}, \thmref{opfullaplace} and \rmref{oppta}) where 
the space discretization is performed using finite element, following closely  the works  in \cite{Wang2014,WangGan2013}  on another exponential integrator scheme.
Our  weak convergence proof does not use Malliavin calculus. Furthermore  we provide some preliminaries results (weak convergence of  the   semi discrete scheme in \thmref{thm:mainadditive} and \thmref{thm:multiplicative})  toward the weak  convergence when the linear operator $A$ 
is not necessarily self-adjoint, and provide some numerical examples to sustain the theoretical results. 
Recent work in \cite{stignonlinearweak} is  used to obtain optimal convergence order for additive noise  when the linear operator is self adjoint in \thmref{opfullaplace} and \rmref{oppta}. 
We also extend in \thmref{strong} the strong optimal  convergence rate  provided in \cite[Theorem 1.1]{kruse} to non-self-adjoint operator $A$. 
 Note that as  the operator $A$ is not necessarily sef-adjoint,  our scheme   here are based on exponential matrix computation. 
 The deterministic part of  this scheme have be proved  to be efficient and robust  comparing to standard schemes in many applications \cite{SebaGatam,AtThesis,TLG,antoelisa} 
 where  the exponential matrix functions have been computed using  the Krylov subspace technique \cite{kry} and fast Leja  points  technique \cite{LE1}.
 For convenience of presentation, we take $A$ to be a second order operator as this simplifies the convergence proof. Our results can be extended to high order semi linear parabolic SPDE.

The  paper is organized as follows. \secref{sec1} provides abstract setting and the well posedness of \eqref{adr}. 
The stochastic exponential Euler scheme along with  weak error representation are provided in \secref{sec2}. The temporal weak convergence rate of the stochastic exponential Euler scheme is provided in  \secref{sec3} for additive noise and in \secref{sec4} for multiplicative noise. 
Note that in this section the linear  operator $A$ is not  necessarily self-adjoint.
\secref{sec5} provides strong optimal convergence rate of the semi discrete solution  for non-self-adjoint operator $A$  along  with full weak convergence rate of  the stochastic exponential  Euler scheme  for self-adjoint operator $A$.
Numerical results to sustain some theoretical results are provided in \secref{sec6}.

\section{The abstract setting and mild solution}
\label{sec1}
Let us start by presenting briefly the notation for the main function
spaces and norms that we use in the paper. 
Let $H $ be a separable Hilbert space  with the  norm $\Vert \cdot \Vert$ associated to
the inner product $\langle,\rangle_H$.
For a Banach space $U$ we denote by $\Vert\cdot\Vert_{U}$ 
the norm of $U$, 
$L(U, H)$  the set of bounded linear mapping  from
$U$ to $H$
and  by $L_{2}(\mathbb{D},U)$\footnote{$\mathbb{D}$ is the sample space} the Hilbert space of all equivalence 
classes of square integrable $U-$valued random variables.  For ease of notation $L(U,U)=L(U)$. 
Furthermore  we denote  by $\mathcal{L}_{1}(U,H)$ the set of  nuclear operators from $U$ to $H$,  $\mathcal{L}_{2}(U,H):=HS(U,H)$ the space of Hilbert Schmidt functions from $U$ to $H$
and $\mathcal{C}_{b}^{k}(U,H)$ the space of not neccessarily bounded mappings from $U$ to $H$ that have continuous and bounded  Frechet  derivatives up to order $k,\, \,k \in \mathbb{N}$.
For simplicity we also write  $\mathcal{L}_{1}(U,U)=\mathcal{L}_{1}(U)$ and $\mathcal{L}_{2}(U,U)=\mathcal{L}_{2}(U)$. 

For a given orthonormal basis $ (e_{i})$ of U, the trace of $ l\in \mathcal{L}_{1}(U)$ is defined by
\begin{eqnarray}
\label{trace}
 \text{Tr}(l):=\underset{i \in \mathbb{N}^{d}}{\sum}\langle l e_i, e_i\rangle_{U},
\end{eqnarray}
while the norm of $ l\in \mathcal{L}_{2}(U)$ is defined by
\begin{eqnarray}
\label{HS}
 \Vert l\Vert_{\mathcal{L}_{2}(U)}^{2}:=\underset{i\in \mathbb{N}^{d}}{\sum}\Vert l e_{i}
 \Vert_{U} ^2 < \infty,
\end{eqnarray}
Note that the trace in\eqref{trace} and the Hilbert Schmidt norm  in \eqref{HS} are independent of the basis $ (e_{i})$.

Let $Q: H\rightarrow H$ be an operator, we  consider throughout this work the $Q$-Wiener process. 
We denote the space of Hilbert--Schmidt operators from 
$Q^{1/2}(H)$ to $H$   by $L_{2}^{0}:=\mathcal{L}_{2}(Q^{1/2}(H),H)= HS(Q^{1/2}(H),H)$ 
and the corresponding norm $\Vert . \Vert_{L_{2}^{0}}$ by
\begin{eqnarray*}
 \Vert l\,\Vert_{L_{2}^{0}} := \Vert l
 Q^{1/2}\Vert_{\mathcal{L}_{2}(H)}=\left( \underset{i \in \mathbb{N}^{d}}{\sum}\Vert
   l Q^{1/2} e_{i} \Vert^{2}\right)^{1/2},\qquad \ l\in L_{2}^{0}.  
\end{eqnarray*}
 Let  $\varphi : [0,T] \times \Omega \rightarrow L_{2}^{0} $  be a  $L_{2}^{0}-$valued 
 predictable stochastic process with 
 $\mathbb{P} \left[ \int_{0}^{T}\Vert \varphi \Vert_{L_{2}^{0}}^{2}ds < \infty \right]=1$.
 Then Ito's isometry (see e.g. \cite[Step 2 in Section 2.3.2]{DaPZ}) gives
 \begin{eqnarray*}
  \mathbb{E} \Vert \int_{0}^{t}\varphi dW \Vert^{2}=\int_{0}^{t}
  \mathbb{E} \Vert \varphi \Vert_{L_{2}^{0}}^{2}ds=\int_{0}^{t}
  \mathbb{E} \Vert \varphi Q^{1/2} \Vert^{2}_{\mathcal{L}_{2}(H)}ds,\qquad \qquad t\in [0,T]. 
 \end{eqnarray*}
Let us recall the following proprieties which will be used in our errors estimation.

\begin{proposition} \cite{Chw}
\label{proposition}
Let $l, l_1, l_2$ be  three operators in Banach spaces, the following proprieties hold 
 \begin{itemize}
  \item If   $l \in  \mathcal{L}_{1}(U)$ then
  \begin{eqnarray}
   \vert \text{Tr}(l) \vert \leq \Vert l \Vert_{\mathcal{L}_{1}(U)}.
  \end{eqnarray}
 \item If $l_1 \in L(H)$ and $ l_2 \in  \mathcal{L}_{1}(H)$, then  both   $ l_1 l_2$ and $ l_2 l_1$ belong to $\mathcal{L}_{1}(H)$ with
 \begin{eqnarray}
 \label{eql}
    \text{Tr}(l_1 l_2)= \text{Tr}(l_2 l_1).
 \end{eqnarray}
 \item If $l_1 \in \mathcal{L}_{2}(U,H)$ and $l_2 \in \mathcal{L}_{2}(H,U)$, then  $ l_1 l_2 \in \mathcal{L}_{1}(H)$ with
 \begin{eqnarray}
  \Vert l_1 l_2 \Vert_{\mathcal{L}_{1}(H)} \leq \Vert l_1 \Vert_{\mathcal{L}_{2}(U,H)} \Vert l_2\Vert_{\mathcal{L}_{2}(H,U)}.
 \end{eqnarray}
 \item If  $l \in \mathcal{L}_{2}(U,H)$, then  its adjoint $l^{*} \in \mathcal{L}_{2}(H,U)$  with 
  \begin{eqnarray}
   \Vert l^{*}\Vert_{\mathcal{L}_{2}(H,U)} = \Vert l \Vert_{\mathcal{L}_{2}(U,H)}.
  \end{eqnarray}
 \item  If $l \in L(U, H)$ and $l_j \in \mathcal{L}_{j}(U),\, j=1,2, $ then $ll_j \in  \mathcal{L}_{j}(U,H)$  with
 \begin{eqnarray}
 \label{eqnn}
  \Vert l l_j \Vert_{\mathcal{L}_{j}(U,H)} \leq \Vert l \Vert_{L (U,H)} \Vert l_j\Vert_{\mathcal{L}_{j}(U)},\;\;\;\,\;\; j=1,2.
 \end{eqnarray}
\end{itemize}
  \end{proposition}
For classical well posedness, some assumptions are required both for the existence and
uniqueness of the solution of equation (\ref{adr}).
%and for our convergence proofs below.
\begin{Assumption}
\label{assumptionn}
The operator $A : \mathcal{D}(A) \subset H \rightarrow H$ is a negative  generator of an analytic semigroup $S(t)=e^{t A}, \quad t\geq 0$.
\end{Assumption}
In the Banach space  $\mathcal{D}((-A)^{\alpha/2})$, $\alpha \in
\mathbb{R}$, we use the notation 
$ \Vert (-A)^{\alpha/2}. \Vert =:\Vert .\Vert_{\alpha} $.
We recall some basic properties of the semigroup $S(t)$ generated by $A$.
\begin{proposition}
 \textbf{[Smoothing properties of the semigroup \cite{Henry}]}\\
\label{prop1}
Let $ \alpha >0,\;\beta \geq 0 $ and $0 \leq \gamma \leq 1$, then  there exist  $C>0$ such that
\begin{eqnarray*}
 \Vert (-A)^{\beta}S(t)\Vert_{L(H)} &\leq& C t^{-\beta}\;\;\;\;\; \text {for }\;\;\; t>0\\
  \Vert (-A)^{-\gamma}( \text{I}-S(t))\Vert_{L(H)} &\leq& C t^{\gamma} \;\;\;\;\; \text {for }\;\;\; t\geq 0.
\end{eqnarray*}
In addition,
\begin{eqnarray*}
%\label{lemmma}
(-A)^{\beta}S(t)&=& S(t)(-A)^{\beta}\quad \text{on}\quad \mathcal{D}((-A)^{\beta} )\\
\text{If}\;\;\; \beta &\geq& \gamma \quad \text{then}\quad
\mathcal{D}((-A)^{\beta} )\subset \mathcal{D}((-A)^{\gamma} ),\\
\Vert D_{t}^{l}S(t)v\Vert_{\beta}&\leq& C t^{-l-(\beta-\alpha)/2} \,\Vert v\Vert_{\alpha},\;\; t>0,\;v\in  \mathcal{D}((-A)^{\alpha/2})\;\; l=0,1, 
\end{eqnarray*}
where $ D_{t}^{l}:=\dfrac{d^{l}}{d t^{l}}$.
\end{proposition}
We describe now in detail the standard assumptions usually used
on the nonlinear terms $F$,$B$ and the noise $W$.
\begin{Assumption}
\label{assumption1}
\textbf{[Assumption on the drift term $F$]}
There exists  a positive constant $L> 0$  such  that
$F: H \rightarrow H$  satisfies the following Lipschitz condition
\begin{eqnarray*}
 \Vert F(Z)- F(Y)\Vert \leq L \Vert Z- Y\Vert \qquad \forall \quad Z, \;  Y  \in H. 
\end{eqnarray*}
\end{Assumption}
As a consequence, there  exists a constant $C>0$  such that
\begin{eqnarray*}
 \Vert F(Z) \Vert &\leq&   \Vert F (0)\Vert + \Vert F (Z) - F (0)\Vert
 \leq \Vert F (0)\Vert + L \Vert Z\Vert \leq C( 1
 +\Vert Z \Vert )\qquad \qquad Z\in H.\\
\end{eqnarray*}
\begin{Assumption}
 \label{assumption2}
\textbf{[Assumption on  the diffusion term $B$]}\\
There exists  a positive constant $L> 0$ such  that the mapping
 $B: H\rightarrow \mathcal{L}_2(H) $ satisfies the following condition
\begin{eqnarray*}
 \Vert B(Z)- B(Y)\Vert_{\mathcal{L}_2(H)} \leq L \Vert Z- Y\Vert \qquad \forall Z,
 Y  \in H.
\end{eqnarray*}
\end{Assumption}
As a consequence, there exists a constant  $C>0$  such that
\begin{eqnarray}
 \Vert B(Z) \Vert_{\mathcal{L}_2(H)} &\leq&   \Vert B(0)\Vert_{\mathcal{L}_2(H)} + \Vert B(Z) - B(0)\Vert_{\mathcal{L}_2(H)} \nonumber \\
 &\leq &\Vert B(0)\Vert_{\mathcal{L}_2(H)} + L \Vert Z\Vert \leq C( 1
 +\Vert Z \Vert )\qquad  Z\in H.\label{noise}
\end{eqnarray}

\begin{theorem}
 \label{existth}
 \textbf{[Existence and uniqueness (\cite{DaPZ})]}\\
Assume that the initial solution $X_{0}$ is an $F_{0}-$measurable $H-$valued random variable and \assref{assumptionn}, \assref{assumption1}, \assref{assumption2} are satisfied.  
There exists a mild solution $X$ to \eqref{adr} unique, up to
equivalence among the processes, satisfying 
\begin{eqnarray}
\mathbb{P} \left( \int_{0}^{T}\Vert X(s)\Vert^{2}ds < \infty\right)=1.
\end{eqnarray}
 For any $p\geq 2$ there exists a constant $C =C(p,T)>0 $ such that 
\begin{eqnarray}
\label{ineq2}
 \underset{t\in [0,T]}{\sup}\mathbb{E} \Vert X(t)\Vert^{p} \leq C\left(1+\mathbb{E} \Vert X_{0} \Vert^{p}\right). 
\end{eqnarray}
 For any $p>2$ there exists a constant $C_{1} =C_{1}(p,T)>0 $ such that 
\begin{eqnarray}
 \mathbb{E}\underset{t\in [0,T]}{\sup} \Vert X(t)\Vert^{p} \leq C_{1}\left(1+\mathbb{E} \Vert X_{0} \Vert^{p}\right). 
\end{eqnarray}
\end{theorem}
The following theorem proves a regularity result of the mild solution $X$ of \eqref{adr}.
\begin{theorem}
 \label{newtheo}
\textbf{[Regularity of  the mild solution (\cite{GTambueexpoM})]}\\ 
Assume that \assref{assumptionn}, \assref{assumption1} and \assref{assumption2} hold. Let $X$ be the  mild solution of (\ref{adr}) given in (\ref{eq1}). If  $X_{0} \in L_{2}(\mathbb{D},\mathcal{D}((-A)^{\beta/2})),\, \beta \in [0,1)$ 
then  for  all  $ t\in [0,T],\,X(t) \in L_{2}(\mathbb{D},\mathcal{D}((-A)^{\beta/2}))$  with 
\begin{eqnarray*}
\label{regsoluion}
 \left(\mathbb{E}\Vert X(t) \Vert_{\beta}^{2}\right)^{1/2}\leq C  \,\left(1+\left(\mathbb{E}\Vert X_{0}\Vert_{\beta}^{2}\right)^{1/2}\right).\\
\end{eqnarray*}
\end{theorem}

For the weak error represention,  we will need the following lemma.
\begin{lemma}[It\^{o}'s formula]\label{lemmaIto}
Let $(\mathbb{D}, \mathcal{F},P; \left\lbrace F_{t}\right\rbrace_{t\geq 0}$) be a filtered probability space. Let
$\phi$ and $\Psi$ be $H-$valued predictable processes, Bochner integrable on $[0,T]$ P-almost surely (see \cite{KovacsLarsson2008}) , and $Y_0$
 be an $F_0$-measurable $H-$valued random variable. Let $G:[0,T]\times H\rightarrow \mathbb{R}$  and assume that the
 Fr\'{e}chet derivatives $G_t(t,x),\, G_x(t,x),$ and $G_{xx}(t,x)$ are uniformly continuous as functions
of $(t,x)$ on bounded subsets of $[0;T]\times H.$ Note that, for
fixed $t,\,\, G_x(t,x)\in L(H,\mathbb{R})$ and we consider $G_{xx}(t,x)$ as an element
of $L(H)$.  Let $W$ be the $Q$-Wiener process.  If $Y$ satisfies
\begin{equation} Y(t) = Y(0) + \displaystyle\int_0^t\phi(s)ds + \displaystyle\int_0^t \Psi(s)dW(s),\end{equation}
then P-almost surely for all $t\in [0,T],$
\begin{equation}
\begin{array}{rcl}G(t,Y(t))&=& G(0,Y(0)) + \displaystyle\int_0^t G_x(s,Y(s))\Psi(s)dW(s) \\
&& +\displaystyle\int_0^t\Big\{G_t(s,Y(s)) + G_x(s,Y(s))\phi(s)
 +\frac{1}{2}\text{Tr}\left(G_{xx}(s,Y(s))\Psi(s)Q^{\frac{1}{2}}(\Psi(s)Q^{\frac{1}{2}})^*\right)\Big\}\,ds \end{array}\end{equation}
\end{lemma}
A proof of this lemma can be found in \cite{DaPZ}.

\section{Application to the second order  semi--linear parabolic SPDE}
\label{sec2}
We assume that $\Omega$  has a smooth boundary or is a convex polygon of $\mathbb{R}^{d},\;d=1,2,3$. 
In the sequel, for convenience of presentation, we take $A$ to be a second order
operator as this simplifies the convergence proof. The result can be extended to high order semi linear parabolic SPDE.

More precisely we take $H=L^{2}(\Omega)$  and  consider the general second order
semi--linear parabolic stochastic partial differential equation given by
\begin{eqnarray}
\label{sadr}
 dX(t,x)=\left(\nabla \cdot \textbf{D} \nabla X(t,x) -\mathbf{q} \cdot \nabla X(t,x) + f(x,X(t,x))\right) dt +b(x,X(t,x))dW(t,x),
\end{eqnarray}
$x \in \Omega, t\in[0,T]$
where $f:\Omega \times\mathbb{R} \rightarrow \mathbb{R}$ is a
globally Lipschitz continuous function  and $b:\Omega \times\mathbb{R} \rightarrow \mathbb{R}$ is a
continuously differentiable function with globally bounded derivatives.
\subsection{The abstract setting for second order semi--linear parabolic SPDE} 
In the abstract form given in \eqref{adr}, the nonlinear functions $F : H \rightarrow H$  and  $ B : H\rightarrow HS(Q^{1/2}(H), H)$ are defined by
\begin{eqnarray}
\label{nemform}
(F(v))(x)=f(x,v(x)),\qquad (B(v)u)(x)=b(x,v(x))\cdot u(x),
\end{eqnarray}
for  all $ x\in \Omega,\;v\in H,\; u \in Q^{1/2}(H)$, with $H=L^{2}(\Omega)$.
Note that we can also define  $ B : H\rightarrow  \mathcal{L}_2(H)$ by 
\begin{eqnarray}
(B(v)u)(x)=b(x,v(x))\cdot Q^{1/2}u(x),
\end{eqnarray}
for  all $ x\in \Omega,\;v\in H,\; u \in H$.

In  order to define rigorously the linear operator, let us set
\begin{eqnarray}\label{eq:A}
 \mathcal{A}&=&\underset{i,j=1}{\sum^{d}}\dfrac{\partial }{\partial x_{i}}\left( D_{i,j}\dfrac{\partial
 }{\partial x_{j}}\right) - \underset{i=1}{\sum^{d}}q_{i}\dfrac{\partial
 }{\partial x_{i}},
\end{eqnarray}
where we assume that  $D_{i,j} \in L^{\infty}(\Omega),\,q_{i}\in L^{\infty}(\Omega)$ and that there exists a positive constant $c_{1}>0$ such that
\begin{eqnarray}
\label{ellipticity}
\underset{i,j=1}{\sum^{d}}D_{i,j}(x)\xi_{i}\xi_{j}\geq c_{1}\vert \xi \vert^{2},  \;\;\;\;\;\;\forall \xi \in \mathbb{R}^{d},\;\;\; x \in \overline{\Omega},\;\;\; c_{1}>0.
\end{eqnarray}
We introduce two spaces $\HH$ and $V$ where $\HH\subset V $  depends
on the choice of the boundary conditions for the SPDE.
For Dirichlet boundary conditions we let 
\begin{eqnarray*}
V= \HH= H_{0}^{1}(\Omega)=\{v\in H^{1}(\Omega): v=0\;\;
\text{on}\;\;\partial \Omega\}, 
\end{eqnarray*}
and for Robin boundary conditions, Neumann boundary being a special case, 
we take $V=  H^{1}(\Omega)$ and 
\begin{eqnarray*}
\HH = \left\lbrace v\in H^{2}(\Omega): \partial v/\partial
  \nu_{\mathcal{A}}+\alpha_{0} v=0\quad \text{on}\quad \partial \Omega\right\rbrace, \qquad \alpha_{0} \in \mathbb{R}.
 \end{eqnarray*}
See \cite{lions} for details.
The corresponding bilinear form of $ -\mathcal{A}$ is given by
\begin{eqnarray}
\label{var}
a(u,v)=\int_{\Omega}\left(\underset{i,j=1}{\sum^{d}} D_{i,j}\dfrac{\partial u}{\partial x_{j}} \dfrac{\partial v}{\partial x_{i}}+\underset{i=1}{\sum^{d}}q_{i} \dfrac{\partial u}{\partial x_{j}}v\right)dx,\;\;\;\;\;\;\; u, v \in V
\end{eqnarray}
for Dirichlet  and Neumann boundary conditions, and by
\begin{eqnarray}
\label{var1}
a(u,v)=\int_{\Omega}\left(\underset{i,j=1}{\sum^{d}} D_{i,j}\dfrac{\partial u}{\partial x_{j}} \dfrac{\partial v}{\partial x_{i}}+\underset{i=1}{\sum^{d}}q_{i} \dfrac{\partial u}{\partial x_{j}}v\right)dx
+\int_{\partial \Omega} \alpha_{0} u\,v\,dx, \;\;\;\;\;\;\; u, v \in V,
\end{eqnarray}
for Robin boundary conditions.
According to G\aa{}rding's inequality (see \cite{AtThesis,lions}),
there exist two positive constants $c_{0}$ and $\lambda_{0}$  such
that 
 \begin{eqnarray}
 \label{coer}
  a(v,v)+c_{0}\Vert v\Vert^{2}\geq  \lambda_{0}\Vert v\Vert_{H^{1}(\Omega)}^{2},\;\;\; \quad \quad \forall v\in V.
\end{eqnarray}

By adding and subtracting $c_{0}X dt$ on the right hand side of (\ref{sadr}), we have a new operator that we still call $\mathcal{A}$  corresponding to the new bilinear form that we still call 
$a$ such that the following  coercivity property holds
\begin{eqnarray}
\label{ellip}
a(v,v)\geq \; \lambda_{0}\Vert v\Vert_{H^{1}(\Omega)}^{2},\;\;\;\;\;\forall v \in V.
\end{eqnarray}
Note that the expression of the nonlinear term $F$ has changed as we include the term $-c_{0}X$ 
in a new nonlinear term that we still denote by $F$.

Note that $a(,)$ is bounded in $V\times V$, so the following  operator $A: V\rightarrow V^{*}$ is well  defined by the Riez's representation theorem
\begin{eqnarray}
\label{opA}
 a(u,v)=-\langle Au, v\rangle,\,\,\,\, \forall u, v \in V,
\end{eqnarray}
where $V^{*}$ is the adjoint space (or dual space) of $V$ and $\langle ,\rangle$ the duality pairing between $V^{*}$ and $V$. By  identifying $H$ to its adjoint space $H^{*}$, 
we get the following  continuous and dense inclusions
\begin{eqnarray}
\label{gerland}
 V\subset H \subset V^{*}.
\end{eqnarray}
So, we have 
\begin{eqnarray}
\label{gerland1}
 \langle u,v\rangle_H =\langle u, v \rangle \qquad \qquad \qquad \forall u\in  H,\, \forall  v\in V.
\end{eqnarray}
The linear operator $A$ in our abstract setting \eqref{adr} is therefore defined by \eqref{opA}.
The domain of $A$ denoted by  $\mathcal{D}(A)$ is defined by 
\begin{eqnarray}
 \mathcal{D}(A)= \{ u \in V,\, Au \in H \}. 
\end{eqnarray}
We write the restriction of $A: V\rightarrow V^{*}$ to $\mathcal{D}(A)$ again by  $A$, which is therefore regarded as an operator of $H$ (more precisely the $H$ realization of $\mathcal{A}$ \cite[p. 812]{lions}).
The coercivity property (\ref{ellip}) implies that $A$ is a sectorial on $L^{2}(\Omega)$ i.e.  there exists $C_{1},\, \theta \in (\frac{1}{2}\pi,\pi)$ such that
\begin{eqnarray}
 \Vert (\lambda I -A )^{-1} \Vert_{L(L^{2}(\Omega))} \leq \dfrac{C_{1}}{\vert \lambda \vert }\;\quad \quad 
\lambda \in S_{\theta},
\end{eqnarray}
where $S_{\theta}=\left\lbrace  \lambda \in \mathbb{C} :  \lambda=\rho e^{i \phi},\; \rho>0,\;0\leq \vert \phi\vert \leq \theta \right\rbrace $ (see \cite{Henry,lions}).
 Then  $A$ is the infinitesimal generator of bounded analytic semigroups $S(t):=e^{t A}$  on $L^{2}(\Omega)$  such that
\begin{eqnarray}
S(t):= e^{t A}=\dfrac{1}{2 \pi i}\int_{\mathcal{C}} e^{ t\lambda}(\lambda I - A)^{-1}d \lambda,\;\;\;\;\;\;\;
\;t>0,
\end{eqnarray}
where $\mathcal{C}$  denotes a path that surrounds the spectrum of $A $.
The condition property \eqref{ellip} also implied that $-A$ is a positive operator and its fractional powers is well defined  for any $\alpha>0,$ by
\begin{equation}
 \left\{\begin{array}{rcl}
         (-A)^{-\alpha} & =& \frac{1}{\Gamma(\alpha)}\displaystyle\int_0^\infty  t^{\alpha-1}{\rm e}^{tA}dt,\\
         (-A)^{\alpha} & = & ((-A)^{-1})^{-1},
        \end{array}\right.
\end{equation}
where $\Gamma(\alpha)$ is the Gamma function (see \cite{Henry}).
%We denote by $\|\cdot\|_\alpha \doteq \|A^{\alpha/2}\|$
%the norm of the space $\mathcal{D}(A^{\alpha/2}).$\\
Functions in $\HH$ satisfy the boundary
conditions and with $\HH$ in hand we can characterize the
domain of the operator $(-A)^{\alpha/2}$ and have the following norm
equivalence (\cite{lions,Stig,ElliottLarsson}) for $\alpha\in \{1,2\}$
\begin{eqnarray*}
\Vert v \Vert_{H^{\alpha}(\Omega)} &\equiv &\Vert (-A)^{\alpha/2} v
\Vert=:\Vert  v \Vert_{\alpha}, \qquad 
\forall v\in \mathcal{D}((-A)^{\alpha/2}),\\
%=  \HH\cap H^{r}(\Omega).
\mathcal{D}((-A)^{\alpha/2})&=&  \HH\cap H^{r}(\Omega)\,\,\qquad \qquad \,\text{ (Dirichlet boundary conditions)},\\
\mathcal{D}(-A)&=&\HH,\,\,\,\qquad\mathcal{D}((-A)^{1/2})= H^{1}(\Omega)\,\,\,\qquad\text{(Robin boundary conditions)}.
\end{eqnarray*} 
\subsection{Numerical scheme and weak error representation}
We consider the discretization of the spatial domain by a finite element method.
Let $\mathcal{T}_{h}$ be a set of disjoint intervals  of $\Omega$
(for $d=1$), a triangulation of $\Omega$ (for $d=2$) or a set of
tetrahedra (for $d=3$) with maximal length $h$ satisfying the usual regularity assumptions \cite{EP,Gyongy2}.

Let $V_{h}\subset V$ denote the space of continuous functions that are
piecewise linear over the triangulation $\mathcal{T}_{h}$. Note that for high order polynomial,  high  order accuracy in space  can be achieved.
To discretize in space we introduce the projection $P_h$ from $L^{2}(\Omega)$  onto $V_{h}$ defined  for  $u
\in L^{2}(\Omega)$ by
\begin{eqnarray}
 (P_{h}u,\chi)=(u,\chi),\qquad \forall\;\chi \in V_{h}.
\end{eqnarray}
The discrete operator $A_{h}: V_{h}\rightarrow V_{h}$ is  defined by
\begin{eqnarray}
( A_{h}\varphi,\chi)=(A\varphi,\chi)=-a(\varphi,\chi),\qquad \varphi,\chi \in V_{h}.
\end{eqnarray}
Like the operator $A$, the discrete operator $A_h$ is also the generator of an analytic semigroup  $S_h:=e^{tA_{h}}$.
Here we  consider the following  semi--discrete form  of the problem (\ref{adr}), which consists to
find the process $X^{h}(t)=X^{h}(.,t) \in V_{h}$ such  that for $t
\in[0, T]$,
\begin{eqnarray}
\label{dadr}
  dX^{h}=(A_{h}X^{h} +P_{h}F(X^{h}))dt + P_{h} B(X^{h}) P_{h}d W,\qquad
  X^{h}(0)=:X_0^h=P_{h}X_{0}.
\end{eqnarray}
Note that \eqref{dadr} is a  finite dimensional stochastic equation.
The  mild solution of (\ref{dadr})  at  time $t_{m}=m \Delta t $ is given by
\begin{eqnarray}
\label{dmild}
  X^{h}(t_{m})&=&S_{h}(t_{m})P_{h}X_{0}+\displaystyle\int_{0}^{t_{m}} S_{h}(t_{m}-s) P_{h}F(X^{h}(s))ds \nonumber \\ && +
 \displaystyle\int_{0}^{t_{m}} S_{h}(t_{m}-s)P_{h}B(X^{h}(s)) P_{h}d W(s),
\end{eqnarray}
where $ \Delta t=T/M,\;\;\; m \in \{0,1,...,M \}, \;\,M \in \mathbb{N}$.

Then, given the mild solution at the time $t_{m}$,  we can construct the corresponding solution at $t_{m+1}$ as
\begin{eqnarray*}
 X^{h}(t_{m+1})&=&S_{h}(\Delta t)X^{h}(t_{m})+\displaystyle\int_{0}^{ \Delta t} S_{h}( \Delta t-s) P_{h}F(X^{h}(s+t_{m}))ds  \\ && +\displaystyle\int_{t_{m}}^{t_{m+1}} S_{h}(t_{m+1}-s)P_{h}B(X^{h}(s)) P_{h}d W(s).
\end{eqnarray*}
To build the numerical scheme, we use the following approximations \cite{GTambueexpoM}
\begin{eqnarray*}
P_{h}  F(X^{h}( t_{m}+s))&\approx&  P_{h}F(X^{h}( t_{m})),\;\;\;\;\; \;\;\;\;\;s \in [0,\; \Delta t],\\
 S_{h}(t_{m+1}-s)P_{h}B(X^{h}(s))&\approx& S_{h}(\Delta t) P_{h}B(X^{h}(t_{m})),\;\;\; s\in [t_{m},t_{m+1}].
\end{eqnarray*}
We can define our approximation $X_{m}^{h}$ of  $X(m \Delta t)$ by
\begin{eqnarray}
\label{new}
 X_{m+1}^{h}=e^{\Delta t A_{h}}X_{m}^{h}+A_{h}^{-1}\left( e^{\Delta t A_{h}}-I\right)P_{h}F(X_{m}^{h})+e^{\Delta t A_{h}}P_{h}B(X_{m}^{h}) P_{h}\left( W_{t_{m+1}}-W_{t_{m}}\right).
\end{eqnarray}
We introduce for brevity the notations
$$S_h(t) ={\rm e}^{tA_h},\quad S_h^{1}(t) = (tA_h)^{-1}({\rm e}^{tA_h}-I), \quad B_h: X \mapsto P_h B(X)P_h.$$
We can then rewrite the scheme (\ref{new}) as
\begin{eqnarray} \label{eq:fvscheme}
 X_{m+1}^{h}=S_h(\Delta t)X_{m}^{h}+\Delta t S_h^1(\Delta t )P_{h}F(X_{m}^{h})+
 S_h(\Delta t)B_h(X_{m}^{h})\left( W_{t_{m+1}}-W_{t_{m}}\right).
\end{eqnarray}
In order to study the weak convergence of the approximation of the solutions we define the functional
\begin{equation}
 \label{eq:mu} \mu^h(t,\psi) = \mathbb{E}\left[\Phi(X^h(t,\psi)\right],\,\,t\in[0,T],\,\psi\in V_h,
\end{equation}
where $X^h(t,\psi)$ is defined by \eqref{dmild} with the initial value $ X_0^h=\psi\in V_h.$ 
It can be shown (see Theorem 9.16 of [9]) that $\mu^h(t,\psi)$ defined by \eqref{eq:mu} is differentiable with respect to $t$ 
and twice differentiable with respect to $\psi,$ and is the unique strict solution of 
\begin{equation}\label{eq:Kolmogorov}
\left\{\begin{array}{rcl}
\frac{\partial \mu^h}{\partial t}(t,\psi) & = & \langle A_h\psi + P_hF(\psi) , D\mu^h(t,\psi)\rangle_H +
\frac{1}{2}\text{Tr}\left[D^2\mu^h(t,\psi)B_h(\psi) B_h (\psi)^*\right]\\
\mu^h(0,\psi)&=& \Phi(\psi),\quad \psi\in V_h.\end{array}\right.\end{equation}

We can here, using the Riesz representation theorem, identify  the first order  derivative of $\mu^h(t,\psi)$ with respect to $\psi,$ denoted as $D\mu^h(t,\psi)$ with an element
of $V_h$ and the second derivative denoted as $D^2\mu^h(t,\psi)$ with a linear operator in $V_h.$  More precisely 
\begin{eqnarray}
\label{derivation}
D\mu^h(t,\psi)(\phi_1)&=&\langle D\mu^h(t,\psi),\phi_1\rangle_H,\,\,\, \forall \,\psi, \phi_1  \in V_h,\\
 D^2\mu^h(t,\psi)(\phi_1,\phi_2) &=& \langle D^2\mu^h(t,\psi)\phi_1,\phi_2\rangle_H, \,\,\forall \,\psi, \phi_1,\phi_2 \in V_h.
\end{eqnarray}
%D^2\mu^h(t,\varphi^h_0(-t)\psi)(\varphi_0^h(-t)\phi_2)\rangle_H,\,\,\psi, \phi_1,\phi_2 \in V_h
%\end{eqnarray}

The following theorem, which is similar to \cite[Theorem 2.2]{Wang2014} is fundamental for our convergence proofs.
\begin{theorem}[Weak error representation formula for the semi discrete problem]
Assume that all the conditions in the assumptions above are fulfilled and let $\{W(t)\}_{t\in[0,T]}$
be a cylindrical $H-$valued $Q-$Wiener process. Then for $\Phi\in\mathcal{C}_b^2(H,\mathbb{R})$ 
the weak approximation error of the scheme in \eqref{eq:fvscheme} has the representation

\begin{equation}\label{eq:errorrep}\displaystyle
\begin{array}{l}
\mathbb{E}[\Phi(X^h_M)]-\mathbb{E}[\Phi(X^h(T))] \\
= \sum\limits_{m=0}^{M-1}\left\{\displaystyle\int_{t_m}^{t_{m+1}}\EE\left[\langle D\mu^h(T-s,\tilde X^h(s)),P_hF(X_m^h)-P_hF(\tilde X^h(s))\rangle_H \right]ds\right.\\
     \hskip 1cm + \frac{1}{2}\EE\displaystyle\int_{t_m}^{t_{m+1}}\text{Tr}\left[D^2\mu^h(T-s,\tilde X^h(s))\left(\left\{S_h(s-t_m)
     B_h(X_m^h)\right\}\left\{S_h(s-t_m)B_h(X_m^h)\right\}^*\right.\right.\\
     \hskip 7cm\left. \left. \left.- B_h(X_m^h)B_h(X_m^h)^*\right)\right]ds \right. \Big\},
\end{array}
\end{equation}
where $\tilde X^h(t)$ is a continuous extension of $X_m^h,$ defined by
\begin{equation}\label{eq:xtilde}
\begin{array}{rcl}
\tilde X^h(t) &= &S_h(t-t_m)X_m^h + (t-t_m)S_h^1(t-t_m)P_hF(X_m^h) \\
&& \hskip 2cm + S_h(t-t_m)B_h(X_m^h)(W(t)-W(t_m)),\quad \text{ for } t\in [t_m, t_{m+1}].\end{array}
\end{equation}
\end{theorem}
\begin{proof}
Introduce the process $\nu^h:[0,T]\times V_h\rightarrow \mathbb{R},$ given by
\begin{equation}\label{eq:nu}
\nu^h(x,\psi) = \mu^h(t,S_h(-t)\psi),\end{equation}
which is twice differentiable with respect to $\psi$ and by previous identifications satisfies
\begin{eqnarray}
\label{deri}
D\nu^h(t,\psi)\phi &=& \langle D\mu^h(t,S_h(-t)\psi),S_h(-t)\phi\rangle_H,\,\,\psi, \phi \in V_h\\
\label{iden}
D^2\nu^h(t,\psi)(\phi_1,\phi_2) &=& \langle D^2\mu^h(t,S_h(-t)\psi)(S_h(-t)\phi_1),
%D^2\mu^h(t,\varphi^h_0(-t)\psi)(
S_h(-t)\phi_2\rangle_H,\,\,\psi, \phi_1,\phi_2 \in V_h.
\end{eqnarray}
One can then check that $\nu^h(t,\psi)$ solves the equation
\begin{equation}\label{eq:pdenu}
\begin{array}{rcl}\frac{\partial\nu^h}{\partial t}(t,\psi) &=& \langle D\nu^h(t,\psi),S_h(t)P_hF(S_h(-t)\psi\rangle_H +
\frac{1}{2}\text{Tr}[D^2\nu^h(t,\psi)S_h(t)B_h(\psi)(S_h(t)B_h(\psi))^*],\\
\nu^h(0,\psi) & =& \Phi(\psi),\quad \psi \in V_h. \end{array}
\end{equation}
Indeed, since $\frac{\partial S_h(-t)\psi}{\partial t} = -A_h S_h(-t)\psi,$  we have that
$$\begin{array}{rcl}\frac{\partial\nu^h}{\partial t}(t,\psi)&=& \frac{\partial\mu^h}{\partial t}(t,S_h(-t)\psi)
+ \langle D\mu^h(t,S_h(-t)\psi),-A_h S_h(-t)\psi\rangle_H.
\end{array}
$$
Using the Kolmogorov equation \eqref{eq:Kolmogorov} and  the fact that $\langle, \rangle_H$ is symmetric  yields
$$\begin{array}{rcl}\frac{\partial\nu^h}{\partial t}(t,\psi)
&= & \langle A_h S_h(-t)\psi + P_hF(S_h(-t)\psi) , D\mu^h(t,S_h(-t)\psi)\rangle_H\\ &&  +
\frac{1}{2}\text{Tr}\left[D^2\mu^h(t,S_h(-t)\psi)S_h(-t)S_h(-t)S_h(t)S_h(t)B_h(\psi)^*\right]\\
&& +  \langle D\mu^h(t,S_h(-t)\psi),-A_h S_h(-t)\psi\rangle_H.\\
\end{array}$$
As  $D^2\mu^h(t,\psi)$  is  identified  as linear operator in $V_h$,  using \propref{proposition}(mainly relation \eqref{eql}), 
the definition of  the trace \eqref{trace}, the expression \eqref{iden}  and  the fact that $\langle, \rangle_H$ is symmetric  allow to have
$$\begin{array}{rcl}\frac{\partial\nu^h}{\partial t}(t,\psi)
&=& \langle D\mu^h(t,S_h(-t)\psi)S_h(-t)\psi,S_h(t) P_hF(S_h(-t)\psi)\rangle_H\\
&& +\frac{1}{2}\text{Tr}\left[D^2\nu^h(t,\psi)S_h(t)B_h(\psi)(S_h(t)B_h(\psi))^*\right]\\
&&= \langle D\nu^h(t,\psi),S_h(t)P_hF(S_h(-t)\psi\rangle_H +
\frac{1}{2}\text{Tr}[D^2\nu^h(t,\psi)S_h(t)B_h(\psi)(S_h(t)B_h(\psi))^*].\end{array}.$$
Now let $Z^h(t) = S_h(T-t)X^h(t),$ using \eqref{dmild} and the notation $X_0^h = P_hX_0,$ we get
\begin{eqnarray}
\label{eq:zh}Z^h(t)&=&S_h(T)X_{0}^h+\displaystyle\int_{0}^{t}S_h(T-s) P_{h}F(X^{h}(s))ds \nonumber \\ && +
 \displaystyle\int_{0}^{t}S_h(T-s)B_h(X^{h}(s))d W(s), \,\,t\in[0,T].
\end{eqnarray}
Also let  $\tilde Z^h(t) = S_h(T-t)\tilde X^h(t),$ using \eqref{eq:xtilde}, we get
\begin{equation}\label{eq:ztilde}
\begin{array}{rcl}
\tilde Z^h(t) &= &S_h(T-t_m)X_m^h + (t-t_m)S_h(T-t)S_h^1(t-t_m)P_hF(X_m^h) \\
&& \hskip 2cm + S_h(T-t_m)B_h(X_m^h) (W(t)-W(t_m)),\quad \text{ for } t\in [t_m, t_{m+1}].\end{array}
\end{equation}
which can also be written as
\begin{equation}\label{eq:ztilde2}
\begin{array}{rcl}
\tilde Z^h(t) &= &S_h(T-t_m)X_m^h + \displaystyle\int_{t_m}^{t}S_h(T-s)P_hF(X_m^h)ds \\
&& \hskip 2cm + \displaystyle\int_{t_m}^tS_h(T-t_m)B_h(X_m^h) dW(s),\quad \text{ for } t\in [t_m, t_{m+1}].\end{array}
\end{equation}
Consequently, we obtain
$$\tilde Z^h(T) = \tilde X^h(t) = X_M^h,\quad Z^h(0) = S_h(T)\tilde X^h(0) = S_h(T)X_0^h.$$
Now the weak error is worked out as
$$\begin{array}{rcl}
  \EE[\Phi(X_M^h)] - \EE[\Phi(X^h(T,X_0^h))] &=& \EE[\Phi(\tilde Z^h(T))] - \EE[\mu^h(T,X_0^h)]\\
   &=& \EE[\nu^h(0,\tilde Z^h(T))]-\EE[\nu^h(T,\tilde Z^h(0))]\\
   &=&\sum\limits_{m=0}^{M-1}\EE[\nu^h(T-t_{m+1},\tilde Z^h(t_{m+1}))]-\EE[\nu^h(T-t_{m},\tilde Z^h(t_{m}))].
  \end{array}$$
  Now using It\^o formula (see Lemma~\ref{lemmaIto}) applied to $G (t,x) = \nu^h(T-t,\tilde Z^h(t))$ in the interval $[t_m, t_{m+1}]$ and the fact that the It\^{o} integral vanishes, we can write
  $$\begin{array}{l}
     \EE[\nu^h(T-t_{m+1},\tilde Z^h(t_{m+1}))]-\EE[\nu^h(T-t_{m},\tilde Z^h(t_{m}))]\\
     = \displaystyle-\int_{t_m}^{t_{m+1}}\EE\left[\frac{\partial \nu^h}{\partial t}(T-s,\tilde Z^h(s))\right]ds
     +\displaystyle\int_{t_m}^{t_{m+1}}\EE\left[ D\nu^h(T-s,\tilde Z^h(s))S_h(T-s)P_hF(X_m^h) \right]ds\\
     \hskip 1cm + \frac{1}{2}\EE\displaystyle\int_{t_m}^{t_{m+1}}\text{Tr}\left[D^2\nu^h(T-s,\tilde Z^h(s))\left\{S_h(T-t_m)
     B_h(X_m^h)\right\}\left\{S_h(T-t_m)B_h(X_m^h)\right\}^*\right]ds.
     \end{array}
     $$
    Using the fact that  $D\nu^h(t,\psi)$  is identified to  an element
of $V_h$ (see the analogue representation at \eqref{derivation})  and \eqref{eq:pdenu}, we finally have
     
     $$\begin{array}{l}
     \EE[\nu^h(T-t_{m+1},\tilde Z^h(t_{m+1}))]-\EE[\nu^h(T-t_{m},\tilde Z^h(t_{m}))]\\
     = \displaystyle\int_{t_m}^{t_{m+1}}\EE\left[\langle D\nu^h(T-s,\tilde Z^h(s)),-S_h(T-s)P_hF(S_h(s-T)\tilde Z^h(s))
     +S_h(T-s)P_hF(X_m^h)\rangle_H\right]ds\\
      \hskip 1cm + \frac{1}{2}\EE\Big\{\displaystyle\int_{t_m}^{t_{m+1}}\text{Tr}\Big[D^2\nu^h(T-s,\tilde Z^h(s))\Big(\left\{S_h(T-t_m)
     B_h(X_m^h)\right\}\left\{S_h(T-t_m)B_h(X_m^h)\right\}^*\\
     \hskip 5cm  - S_h(T-s)B_h(X_m^h)(S_h(T-s)B_h(X_m^h))^*\Big)\Big]\,ds\Big\}   \\
     =  \displaystyle\int_{t_m}^{t_{m+1}}\EE\left[\langle D\mu^h(T-s,\tilde X^h(s)),P_hF(X_m^h)-P_hF(\tilde X^h(s))\rangle_H \right]ds\\
     \hskip 1cm + \frac{1}{2}\EE\Big\{\displaystyle\int_{t_m}^{t_{m+1}}\text{Tr}\Big[D^2\mu^h(T-s,\tilde X^h(s))\Big(\left\{S_h(s-t_m)
     B_h(X_m^h)\right\}\left\{S_h(s-t_m)B_h(X_m^h)\right\}^*\\
     \hskip 7cm  - B_h(X_m^h)B_h(X_m^h)^*\Big)\Big]ds.   \\
    \end{array}$$
\end{proof}
    \section{Weak convergence for a SPDE with additive noise}
\label{sec3}
In this section, we consider the additive noise where $B=Q^{1/2}$. In order to prove our  weak error estimate the following weak assumptions
\footnote{This assumption is weak  compared to \assref{assumption1} and \assref{assumption2}} \cite{Wang2014} will be used. 

\begin{Assumption}\label{ass:driftandB}
\textbf{[Assumption on  nonlinear function $F$, and $Q$ ]}
We assume that $F:H\rightarrow H$ is Lipschitz and twice continuously differentiable and satisfies
\begin{eqnarray}\label{asseq1}
 \|F(X)\|&\leq& L(1+\|X\|),\qquad X \in H,\\\
 \|F'(Z)(X)\|&\leq& L\|X\|,\,\qquad\,X,\,Z \in H,\,\,\,\label{aq1}\,\hskip 4cm\\
 \quad \|F''(Z)(X_1,X_2)&\leq& L\|X_1\|\|X_2\|,\quad   Z,\,X,\,X_1,\,X_2 \in H, \label{aq2}\\
 \|(-A)^{-\beta}F''(Z)(X_1,X_2)\|&\leq& L\|X_1\|\|X_2\|,\,\,Z,X_1, X_2\in H,\text{ for some }\beta\in[0,1),\label{asseq2}\\
 \|(-A)^{-\frac{\gamma}{2}}F'(Z)(X)\|&\leq& L(1+\|Z\|_1)\|X\|_{-1},  \label{asseq3} \\
 && \,\,X\in H,\,\,Z\in \mathcal{D}((-A)^{1/2}) \text{for some} \,\,\gamma\in[1,2)\nonumber.
 \end{eqnarray}
Furthermore, we assume that the covariance operator $Q$ satisfies
\begin{equation}
\|(-A)^{\frac{\beta-1}{2}}Q^{\frac{1}{2}}\|_{\mathcal{L}_2(H)}<\infty,\text{ for some }\beta\in (0,1].\label{asseq4}\end{equation}
% Furthermore, there exists  a positive constant $L> 0$  such  that
% $F: H \rightarrow H$  satisfies the following Lipschitz condition
% \begin{eqnarray*}
%  \Vert F(Z)- F(Y)\Vert \leq L \Vert Z- Y\Vert \qquad \forall \quad Z, \;  Y  \in H.
% \end{eqnarray*}
\end{Assumption}

Note that the semigroup proprieties  in \propref{prop1} are satisfied for the discrete operator $A_h$.  For  our convergence proof, we add  the following propriety to the discrete operator $A_h$.
\begin{proposition} \label{prop:discreteoperators1}
Under  \assref{ass:driftandB} and \assref{assumptionn}, the  following  proprieties are satisfied for discrete operators 
\begin{eqnarray} \label{eq:maincond1}
\|(-A_h)^{\frac{\beta-1}{2}}P_hQ^{\frac{1}{2}}\|_{\mathcal{L}_2 (H)}&<& C + C_\beta h^{1-\beta},\text{ for some }\beta\in (0,1], \label{eq:n1}\\
 %\Vert (-A_h)^{\beta}S_h(t)\Vert_{L(H)} &\leq& C t^{-\beta}\,\,\text {for }\;\;\ \beta\geq 0 \label{eq:1}\\
  %\Vert (-A_h)^{-\gamma}( \text{I}-S_h(t))\Vert_{L(H)} &\leq& C t^{\gamma} \;\;\;\;\; \text {for }\;\;\;
  %\gamma\in[0,1)\label{eq:2}\\
   \|(-A_h)^{-\beta }P_hF''(Z)(X_1,X_2)\|&\leq&  C(1+h^{2\beta}) \|X_1\|\|X_2\|,\,\,Z,X_1, X_2\in V_h;\,\,\beta \in[0,1),\label{eq:nn3}\\ 
   \|(-A_h)^{-\frac{\gamma}{2}}P_hF'(Z)(X)\|&\leq&  C (h^{\gamma} + (1+\|Z\|_1)\|X\|_{-1}),\,\,Z,X\in V_h;\,\,\gamma\in [1,2),\label{eq:n3}\end{eqnarray}
   \end{proposition}
 \begin{proof}
   According to (\cite{lions}, Theorem 5.2) for all $\alpha>0,$ the discrete operator $A_h$ and the continuous operator $A$ satisfy
   \begin{equation} \label{eq:lions}
   \|A_h^{-\alpha}P_h-A^{-\alpha}\|_{L(H)}\leq C_\alpha h^{2\alpha},\end{equation}
   where $C_\alpha$ is a constant dependent on $\alpha.$  For $\alpha=-\frac{\beta-1}{2},$ with $\beta\in (0,1]$, we have 
   $$\|(A_h^{-\alpha}P_h-A^{-\alpha})Q^{\frac{1}{2}}\|_{\mathcal{L}_2(H)} = \|A_h^{-\alpha}P_hQ^{\frac{1}{2}}-A^{-\alpha}Q^{\frac{1}{2}}\|_{\mathcal{L}_2(H)}$$
   Now let $\{e_i\}_{i=1}^\infty$ be any orthonormal basis of $H$, since $\|A^{-\alpha}Q^{\frac{1}{2}}\|_{\mathcal{L}_2(H)}<\infty, $ according to \assref{ass:driftandB} (relation \eqref{asseq4}),
   we have using \eqref{eq:lions}
   $$\begin{array}{rcl}
\|A_h^{-\alpha}P_hQ^{\frac{1}{2}}-A^{-\alpha}Q^{\frac{1}{2}}\|_{\mathcal{L}_2(H)} &=&\sum\limits_{i=1}^\infty \|(A_h^{-\alpha}P_h-A^{-\alpha})Q^{\frac{1}{2}}e_i\|\\
   &\leq &  \|A_h^{-\alpha}P_h-A^{-\alpha}\|_{L(H)}\|Q^{\frac{1}{2}}\|_{\mathcal{L}_2(H)}\\
   &\leq & C_\alpha h^{2\alpha}.\end{array}$$
   Now we can write, using \eqref{eq:lions} and \eqref{asseq4}, that
   $$\begin{array}{rcl}\|A_h^{-\alpha}P_hQ^{\frac{1}{2}}\|_{\mathcal{L}_2(H)} &=& \|A_h^{-\alpha}P_hQ^{\frac{1}{2}}-A^{-\alpha}Q^{\frac{1}{2}}+A^{-\alpha}Q^{\frac{1}{2}}\|_{\mathcal{L}_2(H)}\\
   &\leq & \|A_h^{-\alpha}P_hQ^{\frac{1}{2}}-A^{-\alpha}Q^{\frac{1}{2}}\|_{\mathcal{L}_2(H)} + \|A^{-\alpha}Q^{\frac{1}{2}}\|_{\mathcal{L}_2(H)}\\
   &\leq & C + C_\alpha h^{2\alpha}\\
   &=&C + C_\alpha h^{1-\beta}.\end{array}$$
   This proves \eqref{eq:n1}.
   %The proof of \eqref{eq:1}, \eqref{eq:2} derives directly from the properties of the  semigroup (see \cite{FujitaSuzuki1991}).
   Now to prove \eqref{eq:nn3}, we proceed as above and the fact that \eqref{asseq2} is satisfied.
   Indeed   for $\beta \in[0,1)$,  using \eqref{aq2}, we have
   $$\|(A_h^{-\beta}P_h-A^{-\beta})F''(Z)(X_1,X_2)\| \leq 
   \|A_h^{-\alpha}P_h-A^{-\beta}\|_{L(H)}\| \|F''(Z)(X_1,X_2) \|\leq CL \, h^{2\beta}\,\|X_1\|\|X_2\|.$$
   Therefore,
   $$\begin{array}{rcl}
      \|A_h^{-\beta}P_hF''(Z)(X_1,X_2) \| &\leq& \|(A_h^{-\beta}P_h-A^{-\beta})F''(Z)(X_1,X_2)\|+\|A^{-\beta}F''(Z)(X_1,X_2) \|\\  &\leq& C(1+h^{2\beta}) \|X_1\|\|X_2\|.
     \end{array}$$
     The proof of \eqref{eq:n3} is done just like for \eqref{eq:nn3} using \eqref{asseq3} and \eqref{aq1}. 
     \end{proof} \\
     %\begin{remark}
The estimates in Proposition~\ref{prop:discreteoperators1} are very tight and can influence  the order  of convergence in space and time  when  $\beta$ in \eqref{asseq4} is small.
%In order to concentrate first on the convergence in time, we present now without proof some other estimates pertaining the the discrete operator $A_h$
%and the projection of the drift term. These estimations are motivated by the fact that, since the space domain $\Omega\subset \mathbb{R}^d$ is bounded, we can write 
Indeed  using the fact that 
\begin{equation}
 h < C({\rm meas} (\Omega))^{\frac{1}{d}},
\end{equation}
where ${\rm meas} (\Omega)$ is the either the length, the area or the volume of the domain $\Omega$ and $d$ is the dimension,   we have  the following corollary.
\begin{corollary}\label{estimatesAhb}
 %Assume that the operator $A$ and $S(t)$ satisfies the conditions in
 %Proposition~\ref{prop:semigroupppty} and Assumption~\ref{ass:driftandB}, 
 Under  \assref{ass:driftandB} 
 %and \assref{assumptionn},
  the following discrete proprieties are satisfied 
\begin{eqnarray} \label{eq:maincond}
\|(-A_h)^{\frac{\beta-1}{2}}P_hQ^{\frac{1}{2}}\|_{\mathcal{L}_2(H)}&<& C,\,\, \quad \quad \beta\in (0,1].\\
    \|(-A_h)^{-\beta}P_hF''(Z)(X_1,X_2)\|&\leq&  C \|X_1\|\|X_2\|,\,\,Z,X_1, X_2\in V_h;\,\,\beta\in[0,1),\label{eq:3}\\ 
   \|(-A_h)^{-\frac{\gamma}{2}}P_hF'(Z)(X)\|&\leq&  C (1+\|Z\|_1)\|X\|_{-1},\,\,Z, X\in V_h;\,\,\gamma \in[1,2),\label{eq:9}\end{eqnarray}
   where  the positive constant $C$ is independent of $h$.
 \end{corollary}
     The following lemma will be helpful for the proofs of convergence.
\begin{lemma}
\label{lem}
Assume that all the conditions above are fulfilled and that $\Phi\in \mathcal{C}_2^b(H;\mathbb{R}).$ For
$\gamma\in[0,1],\,\gamma_1,\gamma_2\in[0,1)$ satisfying $\gamma_1 + \gamma_2<1,$ there exists constants $c_\gamma$ and $c_{\gamma_1,\gamma_2}$ such that
\begin{eqnarray}\label{eq42}
\|(-A_h)^{\gamma}D\mu^h(t,\psi)\|&\leq &c_{\gamma}t^{-\gamma},\\
\label{eq43}\|(-A_h)^{\gamma_2}D^2\mu^h(t,\psi)(-A_h)^{\gamma_1}\|_{L(H)}&\leq &c_{\gamma_1,\gamma_2}(t^{-(\gamma_1+\gamma_2)}+1),
\end{eqnarray}
where $\mu^h(t,\psi)$ is defined by \eqref{eq:mu}, and $\psi \in V_h$.

\begin{proof}
 The proof is the same as in \cite{Wang2014}. Note that although the linear operator  in \cite{Wang2014} is assumed to be self adjoint, the proof don't make use of that.
\end{proof}

\end{lemma}
The following results can be proven exactly along the line of \cite[Lemma 3.4 and Lemma 3.5]{Wang2014}).
\begin{lemma}
\label{lem:boundonxmh}
Suppose \assref{assumption1} is fulfilled and $X_0\in \mathcal{D}((-A)^{1/2}),$ then it holds for $\gamma \in [0,\frac{\beta}{2})$ and arbitrary small $\epsilon$ that
\begin{equation}\label{eq44}
\underset{0\leq m \leq M}{\sup} \|(-A_h)^{\gamma}X_m^h\|_{L_2(\mathbb{D}, H)} \leq C \text{ and }
\|\tilde X^h(t)-X^h_m\|_{L_2(\mathbb{D}, H)} \leq C\Delta t^{\frac{\beta-\epsilon}{2}},
\end{equation}
where $X_m^h$ is defined by \eqref{new} and $\tilde X^h$ is given in \eqref{eq:xtilde}.
Furthermore 
\begin{equation}\label{eqbound} \|(-A_h)^{\frac{1}{2}}X_m^h\|_{L_2(\mathbb{D}, H)}\leq C(1+\Delta t^{\frac{\beta-1-\epsilon}{2}}), \end{equation}
for $\beta \in (0,1]$.
%and $m=0,1,\dots,M.$
\end{lemma}
\begin{proof}
Before the proof of  the lemma, we can prove exactly as in \cite[(3.3)]{Wang2014} that
$$
\|X_{k}^{h}\|_{L_2(\mathbb{D}, H)} <\infty, \,\,\qquad 0 \leq k\leq  M,\,\, \underset{\qquad 0 \leq k\leq  M} {\sup} \| P_h F(X_{k}^{h})\|_{L_2(\mathbb{D}, H)} <\infty. 
$$
We concentrate on proving \eqref{eqbound}, the proof of the other assertion can be done just as in \cite{Wang2014}. Recall that
\begin{eqnarray}
 X_{m}^{h}&=&e^{\Delta t A_{h}}X_{m-1}^{h}+A_{h}^{-1}\left( e^{\Delta t A_{h}}-I\right)P_{h}F(X_{m-1}^{h})+e^{\Delta t A_{h}}B_h\left( W_{t_{m}}-W_{t_{m}-1}\right),\\
 B_h&=&P_h Q^{1/2}P_h.
\end{eqnarray}
This can also be written as
\begin{eqnarray}
 X_{m}^{h}=S_h(\Delta t)X_{m-1}^{h}+\displaystyle\int_0^{\Delta t} S_h(\Delta t-s)P_{h}F(X_{m-1}^{h})ds+S_h(\Delta t)B_h\left( W_{t_{m}}-W_{t_{m}-1}\right).
\end{eqnarray}
Iterating gives
\begin{eqnarray}
 X_{m}^{h}=S_h(t_m)X_{0}^{h}+\sum_{k=0}^{m-1}\displaystyle\int_{t_k}^{t_{k+1}} S_h(t_m-s)P_{h}F(X_{k}^{h})ds+\sum_{k=0}^{m-1}S_h(t_m-t_k)B_{h}\Delta W_k.
\end{eqnarray}
Now
$$\begin{array}{rcl}\|(-A_h)^{\frac{1}{2}}X_m^h\|_{L_2(\mathbb{D}, H)}&\leq&
\|(-A_h)^{\frac{1}{2}}S_h(t_m)X_{0}^{h}\|+\sum\limits_{k=0}^{m-1}\displaystyle\int_{t_k}^{t_{k+1}}\|(-A_h)^{\frac{1}{2}} S_h(t_m-s)P_{h}F(X_{k}^{h})\|_{L_2(\mathbb{D},H)}ds \\ && +\left(\Delta t\sum\limits_{k=0}^{m-1}\|(-A_h)^{\frac{1}{2}}S_h(t_m-t_k)B_{h}\|_{\mathcal{L}_2(H)}^2\right)^{\frac{1}{2}}\\
= I_0 + I_1 + I_2.\end{array}$$
We can prove exactly as in (\cite{Wang2014}, Lemma 3.5) that
\begin{equation} \label{eq:resultwang}I_2\leq C\Delta t^{\frac{\beta -1-\epsilon}{2} }\text{ and } I_0\leq \|X_0\|_1<\infty.\end{equation}
Now it remain to prove that $I_1$ is bounded above.
We have
$$\begin{array}{rcl}I_1 &\leq& \sum\limits_{k=0}^{m-1}\displaystyle\int_{t_k}^{t_{k+1}}\|(-A_h)^{\frac{1}{2}} S_h(t_m-s)P_{h}F(X_{k}^{h})ds\|_{L_2(\mathbb{D}, H)}\\
&\leq & C\sum\limits_{k=0}^{m-1}\displaystyle\int_{t_k}^{t_{k+1}}(t_m-s)^{-\frac{1}{2}}(1+\|X_{k}^{h}\|_{L_2(\mathbb{D}, H)})ds\\ & \leq &  
C \left(\sum\limits_{k=0}^{m-1} \displaystyle\int_{t_k}^{t_{k+1}}(t_m-s)^{-\frac{1}{2}}ds\right)\left( \underset{0 \leq k\leq m-1}{\sup}\left(1+\|X_{k}^{h}\|_{L_2(\mathbb{D}, H)}\right) \right)\\
 &=&C \left(\displaystyle \int_{0}^{t_{m}}(t_m-s)^{-\frac{1}{2}}ds\right)\left( \underset{0 \leq k\leq m-1}{\sup}\left(1+\|X_{k}^{h}\|_{L_2(\mathbb{D}, H)}\right) \right)\\
 &\leq& C  T^{1/2} \left( \underset{0 \leq k\leq m-1}{\sup}\left(1+\|X_{k}^{h}\|_{L_2(\mathbb{D}, H)}\right) \right) < \infty.
 \end{array}$$

Combining with \eqref{eq:resultwang} establishes the lemma.\end{proof}\\

Now we can prove the  convergence result for the semi disrete problem  with non-self-adjoint operator $A$ as the time step goes to zero.
\begin{theorem}\label{thm:mainadditive}
Assume that \assref{assumption1} is fulfilled $X_0\in \mathcal{D}((-A)^{1/2})$  and that $\Phi\in \mathcal{C}_2^b(H;\mathbb{R}).$ Then for arbitrary small $\epsilon>0,$
if the adjoint of  the discrete operator $A_h^*$ satisfies
the analogue inequality as $A_h$ in \coref{estimatesAhb}, we have
\begin{equation} |\EE[\Phi(X_M^h)] - \EE[\Phi(X^h(T))]|\leq C\Delta t^{\beta -\epsilon},\end{equation}
where the constant $C$ depends on $\beta, \,\epsilon,\, ...,T,\,L$ and the initial data, but is independent of $h$ and $M.$
\end{theorem}
\begin{proof}
 We derived in \eqref{eq:errorrep} the following error representation
\begin{equation}\label{eq:estimate0}
\mathbb{E}[\Phi(X^h_M)]-\mathbb{E}[\Phi(X^h(T))] = \sum\limits_{m=0}^{M-1}(b_m^1 + b_m^2)
\end{equation}
where the decompositions $b_m^1$ and $b_m^2$ are as follows
\begin{equation}
 \begin{array}{rcl} b_m^1&=&\displaystyle\int_{t_m}^{t_{m+1}}\EE\left[\langle D\mu^h(T-s,\tilde X^h(s)),
 P_hF(X_m^h)-P_hF(\tilde X^h(s))\rangle_{H} \right]ds
     \end{array}
          \end{equation}

          and
\begin{equation}
 \begin{array}{rcl}
     b_m^2&=&  \frac{1}{2}\EE\Big\{\displaystyle\int_{t_m}^{t_{m+1}}\text{Tr}\Big[D^2\mu^h(T-s,\tilde X^h(s))\Big(\left\{S_h(s-t_m)
     B_h\right\}\left\{(S_h(s-t_m)-I)B_h\right\}^*\Big)\Big]ds\Big\}\\
      && +\frac{1}{2}\EE\Big\{\displaystyle\int_{t_m}^{t_{m+1}}\text{Tr}\Big[D^2\mu^h(T-s,\tilde X^h(s))(S_h(s-t_m)-I)B_h B_h^*\Big]ds\Big\}\\
      &=&b_m^{2,1} + b_m^{2,2}.
     \end{array}
          \end{equation}
          
    Note that our term  $b_m^1 $ is  more simple than the one  in \cite{Wang2014}.      
  %Now we estimate $b_m^1.$  
  Applying the Taylor's formula to the drift term $F$ we get 
  $$F(\tilde X^h(s)) - F(X_m^h) = F'(X_m^h)(\tilde X^h(s) - X_m^h) + 
  \int_0^1F''(\chi(r))(\tilde X^h(s)-X_m^h,\tilde X^h(s)-X_m^h)(1-r)dr,$$
  where $\chi(r) = X_m^h + r(\tilde X^h(s)-X_m^h),$ which allows to have 
  
%   \begin{equation} %\label{estimate1}
%    \begin{array}{rcl}
%   |b_m^1|&\leq& 
%   %\Bigg|\displaystyle\int_{t_m}^{t_{m+1}}\EE\Bigg[\Big\langle D\mu^h(T-s,\tilde X^h(s))- D\mu^h(T-s,\tilde X^h(s)),   P_hF(X_m^h)-P_hF(\tilde X^h(s))\Big\rangle_{H} \Bigg]dt \Bigg|\\
%    \Bigg|\displaystyle\int_{t_m}^{t_{m+1}}\EE\Bigg[\Big\langle D\mu^h(T-s,\tilde X^h(s)), 
%    P_hF(X_m^h)-P_hF(\tilde X^h(s))\Big\rangle \Bigg]ds \Bigg|
%    \end{array} \label{eq:estimate1}
%   \end{equation}
  %The last inequality comes from the fact that $\Phi$ and the drift term $F$ are Lipschitz continuous. 
  %Now 
%   \begin{equation}\label{eqJm}
%    \begin{array}{rcl}
%     J_m & =& \Bigg|\displaystyle\int_{t_m}^{t_{m+1}}\EE\Bigg[\Big\langle D\mu^h(T-s,\tilde X^h(s)), 
%    P_hF(X_m^h)-P_hF(\tilde X^h(s))\Big\rangle \Bigg]ds \Bigg|\\
%       \end{array}
%   \end{equation}
%   Applying the Taylor's formula to the drift term $F$ we get 
%   $$F(\tilde X^h(s)) - F(X_m^h) = F'(X_m^h)(\tilde X^h(s) - X_m^h) + 
%   \int_0^1F''(\chi(r))(\tilde X^h(s)-X_m^h,\tilde X^h(s)-X_m^h)(1-r)dr,$$
%   where $\chi(r) = X_m^h + r(\tilde X^h(s)-X_m^h).$  This gives 
  \begin{equation}
   \begin{array}{rcl}
    \lefteqn{|b_m^1|}\nonumber\\
    & \leq & \Big|\displaystyle\int_{t_m}^{t_{m+1}}\EE\Big[\Big\langle_H D\mu^h(T-s,\tilde X^h(s)),
    P_hF'(X_m^h)(\tilde X^h(s) - X_m^h)\Big\rangle_H \Big]ds \Bigg|\\
    && + \Bigg|\displaystyle\int_{t_m}^{t_{m+1}}\EE\Big[\Big\langle D\mu^h(T-s,\tilde X^h(s)), \int_0^1P_hF''(\chi(r))(\tilde X^h(s)-X_m^h,\tilde X^h(s)-X_m^h)(1-r)dr\Big\rangle_H \Big]ds \Bigg|    \\
    &=& J_m^1 + J_m^2.
   \end{array}
  \end{equation}
     
     To estimate $J_m^2,$ we  use \coref{estimatesAhb},\lemref{lem}, 
     %\lemref{lemmadual},
     \lemref{lem:boundonxmh}  
     and  the  Holder inequality 
     \begin{equation} 
   \begin{array}{rcl}
    J_m^2 & \leq &  \displaystyle\int_{t_m}^{t_{m+1}}\EE\Big[\Big|\Big\langle (-A_h)^{-\delta}(-A_h)^\delta D\mu^h(T-s,\tilde X^h(s)), \\ && \hskip 3cm  \displaystyle\int_0^1P_hF''(\chi(r))(\tilde X^h(s)-X_m^h,\tilde X^h(s)-X_m^h)(1-r)dr\Big\rangle_H\Big| \Big]ds\\
    &=& \displaystyle\int_{t_m}^{t_{m+1}}\EE\Big[\Big|\Big\langle (-A_h)^\delta D\mu^h(T-s,\tilde X^h(s)),\\ && \hskip 2cm \displaystyle\int_0^1(-A_h^*)^{-\delta}P_hF''(\chi(r))(\tilde X^h(s)-X_m^h,\tilde X^h(s)-X_m^h)(1-r)dr\Big\rangle_H\Big| \Big]ds\\
    &\leq& \displaystyle\int_{t_m}^{t_{m+1}}\EE\Big[\Big\|(-A_h)^\delta D\mu^h(T-s,\tilde X^h(s))\Big\|\times \\ && \hskip 2cm \Big\| \displaystyle\int_0^1(-A_h^*)^{-\delta}P_hF''(\chi(r))(\tilde X^h(s)-X_m^h,\tilde X^h(s)-X_m^h)(1-r)dr\Big\| \Big]ds\\
    &\leq & c_\delta  \displaystyle\int_{t_m}^{t_{m+1}}\displaystyle\int_0^1\EE\Big[\Big\|(-A_h^*)^{-\delta}P_hF''(\chi(r))(\tilde X^h(s)-X_m^h,\tilde X^h(s)-X_m^h)\Big\|\Big](T-s)^{-\delta}\,dr\,ds\\
    &\leq & c_\delta L \displaystyle\int_{t_m}^{t_{m+1}}\displaystyle\int_0^1\EE\Big[\Big\|\tilde X^h(s)-X_m^h\Big\|^2\big](T-s)^{-\delta}\,dr\,ds\\
    &\leq & C\Delta t^{\beta-\epsilon}\displaystyle\int_{t_m}^{t_{m+1}}(T-s)^{-\delta}\,ds.
   % &\leq & C\Delta t^{\beta -\epsilon}.
    \end{array}\label{eq:estimate2}
    \end{equation}
   % THE CONSTANT L SHOULD BE REPLACE BY $L(1+h^{2\alpha})$
    We now turn to estimate $J_m^1.$ 
    Recall that from \eqref{eq:xtilde}, 
    \begin{equation}\label{eqxtildemx}
     \begin{array}{rcl}
\tilde X^h(s)-X_m^h &= &\big(S_h(s-t_m)-I\big)X_m^h + (s-t_m)\varphi_1^h(s-t_m)P_hF(X_m^h) \\
&& \hskip 2cm + S_h(s-t_m) B_h(W(s)-W(t_m)).\end{array}
    \end{equation}
    Since the expectation of the Brownian motion vanishes, we therefore have
    \begin{equation}
    \begin{array}{rcl}
     J_m^1 &\leq &  \Big|\displaystyle\int_{t_m}^{t_{m+1}}\EE\Big[\Big\langle D\mu^h(T-s,\tilde X^h(s)), P_hF'(X_m^h)\big(S_h(s-t_m)-I\big)X_m^h\Big\rangle_H \Big]ds \Bigg|\\
    && +\Big|\displaystyle\int_{t_m}^{t_{m+1}}\EE\Big[\Big\langle D\mu^h(T-s,\tilde X^h(s)),P_hF'(X_m^h)(S_h^1(s-t_m)P_hF(X_m^h)(s-t_m)\Big\rangle_H \Big]ds \Bigg|\\
    &= & I + II.
    \end{array}
    \end{equation}
Using  \lemref{lem},\coref{estimatesAhb}, \lemref{lem:boundonxmh} and \propref{prop:discreteoperators1} yields 
\begin{equation}
 \begin{array}{rcl}
 \lefteqn{I}\nonumber\\
  &\leq & c_{\frac{\delta}{2}}\displaystyle\int_{t_m}^{t_{m+1}}\EE\Big[
  \Big\|(-A_h^*)^{-\frac{\delta}{2}}P_hF'(X_m^h)\big(S_h(s-t_m)-I\big)X_m^h\Big\|\Big](T-s)^{-\frac{\delta}{2}}ds\\
  &\leq & C\displaystyle\int_{t_m}^{t_{m+1}}\Big[
  (1+\|X^h_m\|_1)|\big\|(S_h(s-t_m)-I)X_m^h\Big\|_1\Big](T-s)^{-\frac{\delta}{2}}ds\\
  &\leq & \displaystyle\int_{t_m}^{t_{m+1}}\Big[ 1+ \EE\|(-A_h)^{1/2}X^h_m\| \Big] \EE\| (-A_h)^{-\frac{1+\beta-\epsilon}{2}}) S_h(s-t_m)-I)(-A_h)^{\frac{\beta-\epsilon}{2}} X_m^h \| (T-s)^{-\frac{\delta}{2}}ds  \\
  &\leq & \displaystyle\int_{t_m}^{t_{m+1}}\Big[1+\Delta t^{\frac{\beta-1-\epsilon}{2}}\Big] (s-t_m)^{\frac{1+\beta-\epsilon}{2}} \EE\|(-A_h)^{\frac{\beta-\epsilon}{2}} X_m^h \| (T-s)^{-\frac{\delta}{2}}ds \\
  &\leq & C\Delta t^{\beta -\epsilon}\displaystyle\int_{t_m}^{t_{m+1}}(T-s)^{-\frac{\delta}{2}}ds.\\
  %&\leq & C\Delta t^{\beta -\epsilon}.
 \end{array}\label{eq:estimate3}
\end{equation}
%TO BE COMPLETED(ANTOINE).
Now we turn to approximate the term $II.$ Note that we can rewrite \eqref{eqxtildemx} as 
$$\begin{array}{rcl}
\tilde X^h(s)-X_m^h &= &\big(S_h(s-t_m)-I\big)X_m^h + \displaystyle\int_{t_m}^s S_h(t-s)P_hF(X_m^h)dt \\
&& \hskip 2cm + S_h(s-t_m)B_h(W(s)-W(t_m)).\end{array}$$
Therefore 
\begin{equation}
 \begin{array}{rcl}
 II&=& \Big|\displaystyle\int_{t_m}^{t_{m+1}}\EE\Big[\Big\langle D\mu^h(T-s,\tilde X^h(s)),
 \displaystyle\int_{t_m}^s P_hF'(X_m^h)S_h(t-s)P_hF(X_m^h)dt\Big\rangle_H \Big]ds \Bigg|\\
 &\leq & \displaystyle\int_{t_m}^{t_{m+1}}\displaystyle\int_{t_m}^s\EE\Big[\|D\mu^h(T-s,\tilde X^h(s))\|
 \| P_hF'(X_m^h)S_h(t-s)P_hF(X_m^h)\|\Big]dtds\\
 &\leq & C \displaystyle\int_{t_m}^{t_{m+1}}\displaystyle\int_{t_m}^s\EE\Big[\|S_h(t-s)P_hF(X_m^h)\|\Big]dtds\\
 &\leq & C\EE[\|P_hF(X_m^h)\|]\displaystyle\int_{t_m}^{t_{m+1}}\displaystyle\int_{t_m}^sdtds\\
 &\leq &  C\Delta t^2.
 \end{array}\label{eq:estimate4}
\end{equation}
Combining \eqref{eq:estimate0}, 
%\eqref{eq:estimate1}
\eqref{eq:estimate3} and \eqref{eq:estimate4}, we have
\begin{eqnarray}
 |b_m^1| \leq  C\Delta t ^{\beta-\epsilon}\displaystyle\int_{t_m}^{t_{m+1}}(T-s)^{-\frac{\delta}{2}}ds +C\Delta t^2\\
 \sum\limits_{m=0}^{M-1} |b_m^1| \leq  C\Delta t ^{\beta-\epsilon}
\end{eqnarray}

The term $b_m^2$ can be approximated just as in \cite{Wang2014} to be 
\begin{equation}\label{eq:estimatebm2}
 \sum\limits_{m=0}^{M-1} |b_m^2|\leq C \Delta t ^{\beta-\epsilon}. 
\end{equation}
 Finally we therefore have
\begin{equation} |\EE[\Phi(X_M^h)] - \EE[\Phi(X^h(T))]|\leq C\Delta t^{\beta -\epsilon}.\end{equation}
 \end{proof}
% \begin{remark}
%  Using Proposition~\ref{prop:discreteoperators1} and following the same ideas as above gives the following convergence result.
%  For any $\epsilon>0,$
%  \begin{equation} |\EE[\Phi(X_M^h)] - \EE[\Phi(X^h(T))]|\leq C(h^{1-\beta}\Delta t^{\alpha -\epsilon}+\Delta t^{\alpha-\epsilon}+h),\end{equation}
%  
%  DISCUSSION ON DIFFERENT VALUES OF BETA AND CONSEQUENCE ON THE ORDER OF CONVERGENCE(ANTOINE)
%\end{remark}
 \section{Weak convergence for a SPDE with multiplicative noise}
 \label{sec4}
  Now we consider the case of a stochastic partial differential equation with multiplicative noise, that is 
  \eqref{adr} 
%   \begin{eqnarray}
%     dX=(AX +F(X))dt + B(X)d W, \qquad  X(0)=X_{0},\qquad t \in [0, T],\qquad T>0
% \end{eqnarray}
% where the noise term $B(X)$ depend explicitly on $X.$ With the numerical scheme given in \eqref{eq:fvscheme}, we proceed just as in the case of additive noise and obtain the error representation formula
% 
% \begin{equation}\label{eq:errorrep2}
% \begin{array}{l}
% \mathbb{E}[\Phi(X^h_M)]-\mathbb{E}[\Phi(X^h(T))] \\
% = \sum\limits_{m=0}^{M-1}\left\{\displaystyle\int_{t_m}^{t_{m+1}}\EE\left[\langle D\mu^h(T-s,\tilde X^h(s)),P_hF(X_m^h)-P_hF(\tilde X^h(s))\rangle \right]ds\right.\\
%      \hskip 1cm + \frac{1}{2}\EE\displaystyle\int_{t_m}^{t_{m+1}}Tr\left[D^2\mu^h(T-s,\tilde X^h(s))\left(\left\{\varphi_0^h(s-t_m)
%      P_hB(X_m^h)\right\}\left\{\varphi_0^h(s-t_m)P_hB(X_m^h)\right\}^*\right.\right.\\
%      \hskip 7cm\left. \left. \left.- P_hB(X_m^h)P_hB(X_m^h)^*\right)\right]ds \right.
% \end{array}
% \end{equation}
% Now we require the noise term $B:H\rightarrow L_2^0=HS(Q^{\frac{1}{2}}(H),H)$ to satisfy the Lipschitz condition 
% \begin{equation}
%  \|B(X)-B(Y)\|_{L_2^0}\leq L\|X-Y\|,\,\, \text{ for all } X,\,Y\in H
% \end{equation}
% which implies the linear growth condition
% \begin{equation}
%  \|B(X) \|_{L_2^0}\leq L(1+\|X\|),\,\, \text{ for all } X\in H.
% \end{equation}

For convergence proof, we make the following  assumption on the noise term.
\begin{Assumption}
\label{noisemu}
We assume  that there exists a constant $\alpha>0$ such that 
 \begin{equation}
  \|(-A)^{\frac{\beta-1}{2}}B(X)\|_{\mathcal{L}_2(H)}\leq C(1+\|X\|_{\beta-1}),\,\, \text{ for all } X\in H,\,\, \beta \in [0,1].
 \end{equation}
\end{Assumption}
One can prove just as in the proof of Proposition~\ref{prop:discreteoperators1} that the discrete operators $A_h$ and $P_hB(X)$ satisfies
 \begin{equation}\label{eq:condmultiplicative}
  \|(-A_h)^{\frac{\beta-1}{2}}P_hB(X)\|_{\mathcal{L}_2(H)}\leq Ch^{1-\beta}(1+\|X\|_{\beta-1})\leq C'(1+\|X\|_{\beta-1}),\,\, \text{ for all } X\in H.
 \end{equation}
 For  $\beta \in [0,1]$, we also have using \propref{proposition}
 \begin{eqnarray}
 \label{eq:condmultiplicative1}
  \|(-A_h)^{\frac{\beta-1}{2}}B_h(X)\|_{\mathcal{L}_2(H)}
  &=&\|(-A_h)^{\frac{\beta-1}{2}}P_hB(X)P_h\|_{\mathcal{L}_2(H)}\nonumber\\
  &\leq& C  \|(-A_h)^{\frac{\beta-1}{2}}P_hB(X)\|_{\mathcal{L}_2(H)}\nonumber\\
  &\leq& C'(1+\|X\|_{\beta-1}), \nonumber\\
  &\leq& C'(1+\|X\|),\,\, \text{ for all } X\in H.
 \end{eqnarray}
 Thanks to Lemma~\ref{lem:boundonxmh} and \eqref{eq:condmultiplicative} (or \eqref{eq:condmultiplicative1}) all the results presented above for additive noise also hold for multiplicative noise and we then have the following convergence result.
 \begin{theorem}
 \label{thm:multiplicative}
  Assume that all the conditions of Theorem~\ref{thm:mainadditive} are satisfied with the condition \eqref{eq:maincond1} replaced by \eqref{eq:condmultiplicative}.
  For  $X_0\in \mathcal{D}((-A)^{1/2}), \quad\Phi\in \mathcal{C}_2^b(H;\mathbb{R})$ and for arbitrary small $\epsilon>0,$  if the adjoint of  the discrete operator $A_h^*$ satisfies
the analogue inequality as $A_h$ in \coref{estimatesAhb}, we have the following weak error convergence rate
\begin{equation} |\EE[\Phi(X_M^h)] - \EE[\Phi(X^h(T))]|\leq C\Delta t^{\beta-\epsilon},\end{equation}
where the constant $C$ depends on $\beta, \,\epsilon,\, ...,T,\,L$ and the initial data, but is independent of $h$ and $M.$
 \end{theorem}
 \begin{proof}
 The proof follows the same lines as that of \thmref{thm:mainadditive} 
 \begin{equation}%\label{eq:estimate0}
\mathbb{E}[\Phi(X^h_M)]-\mathbb{E}[\Phi(X^h(T))] = \sum\limits_{m=0}^{M-1}(b_m^{1} + b_m^{2}).
\end{equation}
 The  term $b_m^{1}$ is the same as in \thmref{thm:mainadditive}.  Here $ b_m^{2}$ is given by
 \begin{equation}
 \begin{array}{rcl}
 b_m^2&=&  \frac{1}{2}\EE \Big[\displaystyle\int_{t_m}^{t_{m+1}}
     \text{Tr}\Big[D^2\mu^h(T-s,\tilde X^h(s))\left (S_h(s-t_m) B_h(X_m^h)\right) \\ 
     && \hskip 5cm \times \left((S_h(s-t_m)-I)B_h(X_m^h)\right)^*\Big]ds\Big]\\

      &&+ \frac{1}{2}\EE\Big [\displaystyle\int_{t_m}^{t_{m+1}}\text{Tr}\Big[D^2\mu^h(T-s,\tilde X^h(s))(S_h(s-t_m)-I) \left(B_h(X_{m}^{h})\right) \left(B_h(X_{m}^{h})\right)^*\Big]ds\Big]\\
      &=&b_m^{2,1} + b_m^{2,2}.
     \end{array}
 \end{equation}
 
 Let us estimate $b_m^{2,1}$. \propref{proposition} allows to have  
\begin{equation}
 \begin{array}{rcl}
%      b_m^{2,1}&=&  \frac{1}{2}\EE \Big[\displaystyle\int_{t_m}^{t_{m+1}}
%      \text{Tr}\Big[D^2\mu^h(T-s,\tilde X^h(s))\left (S_h(s-t_m) P_hB(X_m^h)P_h\right) \\ 
%      && \hskip 5cm \times \left((S_h(s-t_m)-I)P_hB(X_m^h)\right)^*\Big]ds\Big]\\
   \lefteqn{ b_m^{2,1}}\nonumber\\
   &\leq&  \frac{1}{2}\EE\Big [\displaystyle\int_{t_m}^{t_{m+1}}
     \Big\|D^2\mu^h(T-s,\tilde X^h(s))\left [S_h(s-t_m) B_h(X_m^h)\right] \\ 
     && \hskip 5cm\times \left[(S_h(s-t_m)-I)B_h(X_m^h)\right]^*\Big\|_{\mathcal{L}_1(H)}ds\Big]\\
     &= & \frac{1}{2}\EE\Big [\displaystyle\int_{t_m}^{t_{m+1}}
     \Big\|(-A_h)^{\frac{\beta+1}{2}-\epsilon}D^2\mu^h(T-s,\tilde X^h(s))(-A_h)^{\frac{1-\beta}{2}} (-A_h)^{\frac{\beta-1}{2}}\left [S_h(s-t_m) B_h(X_m^h)\right ]\\ 
     && \hskip 5cm\times \left[(-A_h)^{-\frac{\beta+1}{2}+\epsilon}(S_h(s-t_m)-I)B_h(X_m^h)\right ]^*\Big\|_{\mathcal{L}_1(H)}ds\Big]\\
     \end{array}
 \end{equation}
  Using \lemref{lem},\propref{proposition},\propref{prop1} and \eqref{eq:condmultiplicative1} yields the following estimations
  
     \begin{equation}
     \begin{array}{rcl}
     \lefteqn{b_m^{2,1}} \nonumber\\
     &\leq & C \EE\Big [\displaystyle\int_{t_m}^{t_{m+1}}
     \Big\|(-A_h)^{\frac{\beta-1}{2}}S_h(s-t_m) B_h(X_m^h)\left[(-A_h)^{\frac{\beta+1}{2}-\epsilon}(S_h(s-t_m)-I)B_h(X_m^h)\right]^*\Big\|_{\mathcal{L}_1(H)} \\ 
     && \hskip 5cm\times [(T-s)^{-1+\epsilon}+1]ds\Big ]\\    
     &\leq & C\EE\Big [ \displaystyle\int_{t_m}^{t_{m+1}}
     \Big\|(-A_h)^{\frac{\beta-1}{2}}S_h(s-t_m) B_h(X_m^h)\Big\|_{\mathcal{L}_2(H)}\\ && 
     \times \Big\|(-A_h)^{-\frac{\beta+1}{2}+\epsilon}(S_h(s-t_m)-I)B_h(X_m^h)\Big\|_{\mathcal{L}_2(H)} [(T-s)^{-1+\epsilon}+1]ds\Big]\\  
      &\leq & C\EE\Big [\displaystyle\int_{t_m}^{t_{m+1}}
     \Big\|(-A_h)^{\frac{\beta-1}{2}}B_h(X_m^h)\Big\|_{\mathcal{L}_2(H)}\\ && \times \Big\|(-A_h)^{-\beta+\epsilon}(S_h(s-t_m)-I)(-A_h)^{\frac{\beta-1}{2}}B_h(X_m^h)\Big\|_{\mathcal{L}_2(H)} [(T-s)^{-1+\epsilon}+1]ds\Big]\\ 
     &\leq & C\Big [\displaystyle\int_{t_m}^{t_{m+1}}
     \underset { 0 \le i \leq M }{\sup}\EE(1+\|X_i^h\|)^2\cdot \Big\|(-A_h)^{-\beta+\epsilon}(S_h(s-t_m)-I)\Big\|_{L(H)} [(T-s)^{-1+\epsilon}+1]ds\Big]\\      
      &\leq & C\Delta t^{\beta-\epsilon}\displaystyle\int_{t_m}^{t_{m+1}}
     (T-s)^{-1+\epsilon}ds + C\Delta t^{\beta-\epsilon+1 }.
     \end{array}
 \end{equation}
 Similar as for $b_m^{2,2}$, we have
 \begin{equation}
     \begin{array}{rcl}
     \lefteqn{b_m^{2,2}} \nonumber\\
     & =& \frac{1}{2}\EE\Big [\displaystyle\int_{t_m}^{t_{m+1}}\text{Tr}\Big[ (-A_h)^{\frac{1-\beta}{2}}D^2\mu^h(T-s,\tilde X^h(s))(-A_h)^{\frac{\beta+1}{2}-\epsilon} \\
    &&  \times (-A_h)^{-\frac{\beta+1}{2}+\epsilon}(S_h(s-t_m)-I) \left(B_h(X_{m}^{h}) \right) \left((-A_h)^{\frac{\beta-1}{2}} B_h(X_{m}^{h})\right)^*\Big]ds\Big]\\
      &\leq & C \EE\Big [\displaystyle\int_{t_m}^{t_{m+1}}
     \Big\|(-A_h)^{-\beta+\epsilon}\left(S_h(s-t_m)-I\right) (-A_h)^{\frac{\beta-1}{2}}B_h(X_m^h)\left[(-A_h)^{\frac{\beta-1}{2}}B_h(X_m^h)\right]^*\Big\|_{\mathcal{L}_1(H)} \\ 
     && \hskip 5cm\times [(T-s)^{-1+\epsilon}+1]ds\Big ]\\  
    &\leq& C\Delta t^{\beta-\epsilon}\displaystyle\int_{t_m}^{t_{m+1}}
     (T-s)^{-1+\epsilon}ds + C\Delta t^{\beta-\epsilon+1}.
     \end{array}
\end{equation}
So  by summing up  as in \thmref{thm:mainadditive}, the proof is ended.
 \end{proof}
\section{Strong convergence and toward full weak convergence results} 
\label{sec5}
 The goal here is to provide the space and time  convergence proof of the exponential scheme. Before  that we will provide strong convergence of  the semi discrete solution.
\begin{theorem}
\label{strong}
 Let $X^h$ and  $\overline{X}_h$ be the solutions respectively of \eqref{dadr}  and  the following  semi discrete problem
\begin{eqnarray}
\label{dadrnt}
  d\overline{X}_h&=&(A_{h}\overline{X}_h +P_{h}F(\overline{X}_h))dt + P_{h} B(\overline{X}_h)d W\\
  \overline{X}_h(0)&=&P_{h}X_{0} \nonumber.
\end{eqnarray}
 Let  $ \beta \in [0,2)$, 
 assume that \assref{assumption1} and \assref{assumption2}
 %all the conditions of Theorem~\ref{thm:mainadditive}  and \thmref{thm:multiplicative} 
 are satisfied. For  $\beta \in [0,1]$  assume  that the relation \eqref{asseq4} of \assref{ass:driftandB} (when dealing with additive noise) and \assref{noisemu}
 (when dealing with  multiplicative noise) are satisfied.  For   $\beta \in [1,2)$  assume  that $B(\mathcal{D}((-A)^{\frac{\beta-1}{2}}))\subset HS\left(Q^{1/2}(H),\mathcal{D}((-A)^{\frac{\beta-1}{2}})\right)$
and $\Vert (-A)^{\frac{\beta-1}{2}}B(v)\Vert_{L_{0}^{2}}\leq c(1+\Vert v\Vert_{\beta-1})$ for $v \in \mathcal{D}((-A)^{\frac{\beta-1}{2}})$. 
If  $X_0 \in L_2(\mathbb{D},\mathcal{D}((-A)^{\beta/2})) $,  there exists a positive constant $C$ independent of $h$ such that the following  estimations hold: 
 \begin{eqnarray}
 \label{noseleq}
 \Vert X(t)-\overline{X}_{h}(t)\Vert_{L_2(\mathbb{D},H)} \leq Ch^{\beta},\;\;\,\, \beta \in [0,1].
 \end{eqnarray}
 Furthermore assume  that the linear operator $A$ is self adjoint, the following  estimation  hold
 \begin{eqnarray}
 \label{seleq}
 \Vert X(t)-X^{h}(t)\Vert_{L_2(\mathbb{D},H)} \leq Ch^{\beta}, \,\;\; \beta \in [0,2).
 \end{eqnarray}
 %where $X_{h}$ is  the solution of \eqref{dadr}.
 \end{theorem}
 Before   prove \thmref{strong} let us make some remarks and provide some preparatory results.
 \begin{remark}
  \thmref{strong} extends \cite[Theorem 1.1]{kruse} for non-self-adjoint operator $A$ and also provide optimal convergence proof for $\beta \in [0,1)$ which was not studied in \cite{kruse}.
 \end{remark}

 \begin{remark}
 \label{troncate}
 For additive noise, we can observe  from \cite{stigstrong} that the semi discrete problem \eqref{dadr} is equivalent to the following problem, find $X_h(t) \in V_h$ such that
  \begin{eqnarray}
\label{dadrntt}
  d X_h&=&(A_{h} X^h +P_{h}F(X^h))dt + d W_h\\
  X^h(0)&=&P_{h}X_{0} \nonumber.
\end{eqnarray}
where  $W_h(t)$ is  a $P_{h}QP_{h}$-Wiener process on $V_h$ with the  following representation
\begin{eqnarray}
  \label{eq:Wh}
  W_h(t)=\underset{ i=1}{\sum^{N_h}}\sqrt{q_{h,i}}e_{h,i}\beta_{i}(t),
\end{eqnarray}
where $(q_{h,i},e_{h,i})$ are the eigenvalues and  eigenfunctions
of the covariance operator $Q_h:=P_{h}QP_{h}$ and $\beta_{i}$ are independent and identically distributed
standard Brownian motions. More precisely $(q_{h,i},e_{h,i})$ is the finite element solution of the eigenvalue problem $Qu=\gamma u$. 

If  the exact eigenvalues and eigenfunctions  $(q_{i},e_{i})$  of  the covariance operator $Q$ are known, replacing  in \eqref{dadrntt} (or \eqref{dadr}) $W_h(t)$ by $W_h^{N_h}(t)$, defined  by
\begin{eqnarray}
  \label{eq:Whh}
  W_h^{N_h}(t)=\underset{ i=1}{\sum^{N_h}}\sqrt{q_{i}}e_{i}\beta_{i}(t),
\end{eqnarray}
will not necessarily  change the optimal convergence order in our scheme. From  \cite{stigstrong}, it is also proved that if  the  kernel of the covariance function $Q$ is regular
and  the mesh family is quasi-uniform, it is enough to take $M < N_h$  noise terms in  \eqref{eq:Whh} (or \eqref{eq:Wh}) without loss the optimal convergence order.

Of course for  multiplicative noise $P_h W$ can  also be expanded  on the basis of $V_h$ with $N_h$ terms (see \cite{Yn:05}).
\end{remark}

 As we are also dealing in \thmref{strong} with non-self-adjoint operator in  \eqref{noseleq}, let us provide some  preparatory results
 before giving  the proof  of \thmref{strong}.
 
 We introduce  the Riesz representation operator  $R_{h}: V \rightarrow V_{h}$ defined by 
\begin{eqnarray*} 
 (-A R_{h}v,\chi)=(-A v,\chi)=a(v,\chi),\qquad \qquad v \in V,\; \forall \chi \in V_{h}.
\end{eqnarray*}
Under the usual regularity assumptions on the triangulation
% (for homogenous Dirichlet  boundary condition, see \cite{vidar}), that should be in our case 
% \begin{eqnarray}
%  \underset{\chi \in \mathcal{T}_{h}}{ \inf} \left[ \Vert v-\chi \Vert +h \Vert \nabla ( v-\chi)\Vert \right] \leq  C h^{r}  \Vert v\Vert_{r}, \qquad  v \in \mathcal{D}(A^{r/2}),\; \; 1 \leq r \leq 2, 
% \end{eqnarray}
and in view of $V-$ellipticity \eqref{ellipticity},  it is well known (see \cite{lions}) that the following error bounds holds
\begin{eqnarray}
\label{regulard}
  \Vert R_{h}v-v\Vert +h \Vert
    R_{h}v-v\Vert_{H^{1}(\Omega} \leq  C h^{r} \Vert
  v\Vert_{H^{r}(\Omega)}, \qquad v\in V\cap
  H^{r}(\Omega),\; \; r \in \{1,2\}. 
\end{eqnarray}
By interpolation, we have 
\begin{eqnarray}
\label{regulard}
  \Vert R_{h}v-v\Vert +h \Vert
    R_{h}v-v\Vert_{H^{1}(\Omega} \leq  C h^{r} \Vert
  v\Vert_{r}, \qquad v \in \mathcal{D}(A^{r/2}),\; \; 1 \leq r \leq 2. 
\end{eqnarray}
Let us consider  the following deterministic problem, which  consists of finding $u \in V$ such that such that  
\begin{eqnarray}
\label{homog}
u'=Au \qquad 
\text{given} \quad u(0)=v,\qquad  t\in (0,T] .
\end{eqnarray}
The corresponding semi-discretization in space is : Find $u_{h} \in
V_{h}$ such that  
$$u_{h}'=A_{h}u_{h}$$ 
where $u_{h}^{0}=P_{h}v$.
Define the operator 
\begin{eqnarray}
\label{form1}
T_{h}(t) :=  S(t)-S_{h}(t) P_{h} = e^{tA} - e^{tA_h}P_h
\end{eqnarray} 
so that $u(t)-u_{h}(t)= T_{h}(t) v$.

The following lemma will be important in our proof.
\begin{lemma}
\label{lemme11}
The following estimates hold on the semi-discrete approximation of\eqref{homog}. There exists  a constant $C>0$ such that 
\begin{itemize}
 \item (i)  For  $v \in \mathcal{D}((-A)^{\gamma/2})$
 \begin{eqnarray}
 \label{form4}
 \Vert u(t)-u_{h}(t)\Vert &=&\Vert T_{h}(t) v\Vert \leq C h^{r} t^{-(r-\gamma)/2}\Vert v \Vert_{\gamma},\;1\leq r  \leq 2, \, \;\; 0\leq \gamma \leq r.
\end{eqnarray}
\item (ii) For  $v \in \mathcal{D}((-A)^{(\gamma-1)/2})$
\begin{eqnarray}
 \left(\int_0^{t} \Vert T_{h}(s) v\Vert^{2}ds \right)^{\frac{1}{2}} \leq Ch^{\gamma}\Vert v\Vert_{\gamma-1},\,\, 0 \leq \gamma\leq 2.
\end{eqnarray}
\end{itemize}
%where $r$ arises in \eqref{regulard}.
\end{lemma}
\begin{proof}
 The  proof of (i) can be found in \cite{GTambueexpoM} using \eqref{regulard}. For self adjoint operator, the proof of (ii) is done as in  \cite[Lemme 4.1]{Yn:04} if $\gamma \in [0,1]$ 
 and in\cite[Lemma 4.2]{kruse}  if $\gamma \in [1,2)$ and  the parameter $r$
 used in \cite[Lemma 4.2]{kruse} is $r=\gamma-1$. Both proofs only uses general concepts and not spectral decomposition of  the linear operator $A$, so can easily be generalized. Note 
 that the non-self adjoint case should make  use 
 of  \cite[4.17]{vidar} or \cite[Lemma 4.3]{vidar}  instead of \cite[2.29]{vidar} used in \cite[Lemma 4.2]{kruse}.
 
\end{proof}

Let us now provide the proof of \thmref{strong}.

\begin{proof}
 %Set 
%\begin{eqnarray*}
% X(t_{m})&=&S(t_{m})X_{0}+\underset{k=0}{\sum^{m-1}}\int_{t_{k}}^{t_{k+1}}S(t_{m}-s)F(X(s))ds+\int_{0}^{t_{m}} S(t_{m}-s)B(X(s))d W(s)\\
 %         &=& \overline{X}(t_{m})+ O(t_{m}).
%\end{eqnarray*}
The proof of the estimation  \eqref{seleq} for additive noise can be found  in   \cite{stigstrong} and can  be updated to multiplicative noise following \cite{Yn:05}.

For $\beta \in [1,2)$, the proof for the estimation \eqref{noseleq} when $A$  is self adjoint operator  can be found in \cite{kruse} as the \assref{assumption1} implies \cite[Assumption 2.1]{kruse} 
by taking  $r=\beta-1$ in \cite[Theorem 1.1]{kruse}.   
Let us give more general proof by  closely follow  \cite{stigstrong,kruse}.
The corresponding mild solution  of \eqref{dadrnt}is given by
\begin{eqnarray}
\overline{X}_h(t)= S_h(t) P_hX_0+ \int_{0}^{t}S_h(t-s)P_hF(\overline{X}_h(s))ds +\int_{0}^{t} S_h(t-s)P_h B(\overline{X}_h(s))d W(s).
\end{eqnarray}

Indeed, we have
 \begin{eqnarray}
 \lefteqn{\Vert X(t)-\overline{X}_{h}(t)\Vert_{L_2(\mathbb{D},H)}} \nonumber\\
 &\leq&  \Vert  T_h(t)P_hX_0+ \int_{0}^{t}S(t-s)F(X(s))ds -\int_{0}^{t}S_h(t-s)P_hF(\overline{X}_h(s))ds\Vert_{L_2(\mathbb{D},H)} \nonumber\\
 && + \Vert \int_{0}^{t} S(t-s)B(X(s))d W(s)-\int_{0}^{t} S_h(t-s)P_h B(\overline{X}_h(s))d W(s) \Vert_{L_2(\mathbb{D},H)}\nonumber\\
 &=& I_1 +I_2.
\end{eqnarray}
The estimation of $I_1$ is  the same as in \cite{stigstrong} and we have 
\begin{eqnarray}
 I_1 \leq C \int_{0}^{t} \Vert X(s)-\overline{X}_{h}(s)\Vert_{L_2(\mathbb{D},H)} ds +C h^{\beta}.
\end{eqnarray}
For  the estimation of $I_2$, we follow closely \cite{kruse}. Indeed we have 
\begin{eqnarray}
 I_2 &\leq& C \Big ( \mathbb{E} \Big [  \int_{0}^{t}  \Vert S(t-s)B(X(s))-S_h(t-s)P_h B(\overline{X}_{h}(s)) \Vert_{\mathcal{L}_2(H)}^2 ds  \Big ] \Big)^{1/2}\\
    &\leq&  C \Big \Vert \Big ( \int_{0}^{t}  \Vert S_h(t-s)P_h ( B(\overline{X}_{h}(s))-B(X(s)))\Vert_{\mathcal{L}_2(H)}^2 ds  \Big )^{1/2}  \Big \Vert_{L_2(\mathbb{D},\mathbb{R})}\\
    &&+ C \Big \Vert \Big ( \int_{0}^{t}  \Vert T_h(t-s)( B(X(s))-B(X(t)))\Vert_{\mathcal{L}_2(H)}^2 ds  \Big )^{1/2}  \Big \Vert_{L_2(\mathbb{D},\mathbb{R})}\\
    &&+ C \Big \Vert \Big ( \int_{0}^{t}  \Vert T_h(t-s)B(X(t)\Vert_{\mathcal{L}_2(H)}^2 ds  \Big )^{1/2}  \Big \Vert_{L_2(\mathbb{D},\mathbb{R})}\\
     &=& I_2^{1}+  I_2^{2}+  I_2^{3}.   
\end{eqnarray}
The stability propriety of  the semi group \propref{prop1} and the Lipschitz condition in \assref{assumption1} allow to have
\begin{eqnarray}
 I_2^{1} \leq  C \left( \int_{0}^{t} \Vert X(s)-\overline{X}_{h}(s)\Vert_{L_2(\mathbb{D},H)}^2 ds\right)^{\frac{1}{2}}.
\end{eqnarray}
Following closely \cite{kruse}, but with  \eqref{form4} in \lemref{lemme11} 
with $r= \beta, \, \gamma=0$,
allow to have
\begin{eqnarray}
  I_2^{2} \leq C h^{\beta}.
\end{eqnarray}
%for $\epsilon$ small enough.
% % For $ \beta[0,1]$  we follow closely \cite{AntoineGab}.
% % Indeed using \propref{proposition} (precisely \eqref{eqnn}) and \eqref{noise},  we also have
% % \begin{eqnarray}
% %  I_2^{3}&=& C \Big \Vert \Big ( \int_{0}^{t}  \Vert T_h(t-s)B(X(t)\Vert_{L_2^0}^2 ds  \Big )^{1/2}  \Big \Vert_{L_2(\mathbb{D},\mathbb{R})}\\
% %    &\leq& C \Big \Vert \Big ( \int_{0}^{t}  \Vert T_h(t-s)\Vert_{L(H)}^2  \Vert B(X(t)\Vert_{L_2^0}^2 ds  \Big )^{1/2}  \Big \Vert_{L_2(\mathbb{D},\mathbb{R})}\\
% %     &\leq& C \Big \Vert \int_{0}^{t}  \Vert T_h(t-s) (-A)^{\frac{1-\beta}{2}}\Vert_{L(H)}^2 ds  \Big )^{1/2} \left(1 +\underset{ 0\leq s \leq T}{\sup} \Vert X(s)\Vert_{L_2(\mathbb{D},H)}\right) \\
% % \end{eqnarray}
% % Using relation   \eqref{form4} in \lemref{lemme11} 
% % with $r= \beta, \, \gamma=0$ yields
% % For $v \in H=L^{2}(\Omega)$ we obviously have  $(-A)^{-\alpha/2} v \in \mathcal{D}((-A)^{\alpha/2})$. Using \lemref{lemme1} yields 
% % \begin{eqnarray*}
% %  \Vert T_{h}(t_{m}-t_{k})(-A)^{-\alpha/2}\Vert_{L(L^{2}(\Omega))}^{2}&=&\underset{v\neq 0} {\sup} \dfrac{\Vert T_{h}(t_{m}-t_{k})(-A)^{-\alpha/2}v \Vert }{\Vert v\Vert} \\
% %     &\leq & C  h^{2} (t_{m}-t_{k})^{-1+\alpha}\underset{v\neq 0} {\sup} \dfrac{\Vert (-A)^{-\alpha/2}v \Vert_{\alpha} }{\Vert v\Vert} \\                                                        
% % &= & C  h^{2} (t_{m}-t_{k})^{-1+\alpha}\underset{v\neq 0} {\sup} \dfrac{\Vert(-A)^{\alpha/2} (-A)^{-\alpha/2}v \Vert}{\Vert v\Vert}\\ \\                    
% % &= & C  h^{2} (t_{m}-t_{k})^{-1+\alpha}.
% % \end{eqnarray*}

For $\beta \in [0,2)$,  as in \cite{kruse}, by using  (ii) in \lemref{lemme11}  gives
\begin{eqnarray}
 I_2^{3} \leq  C h^{\beta}.
\end{eqnarray}

Coming back to $I_2$, we have 
\begin{eqnarray}
 &&\Vert \int_{0}^{t} S(t-s)B(X(s))d W(s)-\int_{0}^{t} S_h(t-s)P_h B(\overline{X}_h(s))d W(s) \Vert_{L_2(\mathbb{D},H)} \\
 &\leq&   C h^{2-\epsilon} +C \left( \int_{0}^{t} \Vert X(s)-\overline{X}_{h}(s)\Vert_{L_2(\mathbb{D},H)}^2 ds\right)^{\frac{1}{2}}\\
 &\leq&   C h^{\beta} +C \left( \int_{0}^{t} \Vert X(s)-\overline{X}_{h}(s)\Vert_{L_2(\mathbb{D},H)}^2 ds\right)^{\frac{1}{2}}.
\end{eqnarray}
Combining $I_1$ and $I_2$  gives
\begin{eqnarray}
 \Vert X(t)-\overline{X}_{h}(t)\Vert_{L_2(\mathbb{D},H)}^2  &\leq&  Ch^{2 \beta} +C \int_{0}^{t} \Vert X(s)-\overline{X}_{h}(s)\Vert_{L_2(\mathbb{D},H)}^2 ds.
\end{eqnarray}
Gronwall's lemma  is therefore applied to end the proof.
\end{proof}

The following theorem provide the full weak convergence when the solution is regular enough.
\begin{theorem}
\label{fullweak}
 Let $X$  and $X_M^h$ be respectively the solution of \eqref{adr} and the numerical solution from \eqref{eq:fvscheme} at the final time $T$.
 Let  $ \beta \in [1,2)$, assume that \assref{assumption1} (for multiplicative noise, all conditions except \eqref{asseq4}) and \assref{assumption2} (for multiplicative noise)
 %all the conditions of Theorem~\ref{thm:mainadditive}  and \thmref{thm:multiplicative} 
 are satisfied. For $\beta=1$ (trace class noise)  assume \eqref{asseq4} of \assref{ass:driftandB} (when dealing with additive noise) and \assref{noisemu}
 (when dealing with  multiplicative noise) are satisfied.
   For   $\beta \in (1,2)$  assume  that  \eqref{asseq4} of \assref{ass:driftandB} is also satisfied for additive noise, 
  but  $B(\mathcal{D}((-A)^{\frac{\beta-1}{2}}))\subset HS\left(Q^{1/2}(H),\mathcal{D}((-A)^{\frac{\beta-1}{2}/2})\right)$
and, $\Vert (-A)^{\frac{\beta-1}{2}}B(v)\Vert_{L_{0}^{2}}\leq c(1+\Vert v\Vert_{\beta-1})$ for $v \in \mathcal{D}((-A)^{\frac{\beta-1}{2}})$ for multiplicative noise. 
 Furthermore assume that $X_0\in \mathcal{D}((-A)^{\beta/2})$ and the linear operator is selfadjoint.
 For $\Phi\in \mathcal{C}_2^b(H;\mathbb{R})$ and  arbitrary small $\epsilon>0,$ the following estimation hold
\begin{equation} |\EE[\Phi(X_M^h)] - \EE[\Phi(X(T)]|\leq C(\Delta t^{1-\epsilon}+h^\beta),
\end{equation}
where the constant $C$ depends on $\alpha, \,\epsilon,\, ...,T,\,L$ and the initial data, but is independent of $h$ and $M.$ 
\end{theorem}
\begin{proof}
 Indeed we have  the following decomposition
 \begin{eqnarray}
  |\EE[\Phi(X_M^h)] - \EE[\Phi(X(T)]| \leq|\EE[\Phi(X_M^h)] - \EE[\Phi(X^h(T)]|+|\EE[\Phi(X^h(T)] - \EE[\Phi(X(T)]|.
  \end{eqnarray}
  Note that if  $X_0\in \mathcal{D}((-A)^{\beta/2}),\, 1 \leq \beta< 2$, then $X_0\in \mathcal{D}((-A)^{1/2})$
  and all the lemma used in \thmref{thm:mainadditive} and \thmref{thm:multiplicative} are still valid.
From \thmref{thm:mainadditive} and \thmref{thm:multiplicative} with $\beta=1$ (as the noise is assumed to be trace class), we have
\begin{eqnarray}
  |\EE[\Phi(X_M^h)] - \EE[\Phi(X^h(T)]|\leq C\Delta t^{1-\epsilon}.
\end{eqnarray}
Note  that this temporal order is  the double of  the strong convergence  obtained in \cite{GTambueexpoM}.
Remember that we are using low order finite element method for space discretization, where  the optimal convergence  is 2 (for deterministic case), 
therefore for  $X_0\in \mathcal{D}((-A)^{\beta/2}),\, 1 < \beta< 2$, 
 according  to  \thmref{strong} and \rmref{troncate}, we cannot expect the  order of weak convergence to double the  strong order in space. The weak convergence order in space (of course, not necessarily optimal) 
 can be the same  as  the strong convergence order when  the solution is regular enough.
 Using the fact  that $\Phi\in \mathcal{C}_2^b(H;\mathbb{R})$, so is Lipschitz, we therefore have
 \begin{eqnarray*}
  |\EE[\Phi(X^h(T)] - \EE[\Phi(X(T)]|&\leq& |\EE[\Phi(X^h(T))- \Phi(X(T))]| \\
  &\leq&  \EE \Vert X^h(T)-X(T)\Vert \\
  & \leq & C \Vert X^h(T)-X(T)\Vert_{L_2(\mathbb{D},H)}\\
  & \leq & Ch^{\beta}.
 \end{eqnarray*}
\end{proof}

The following  theorem  using recent result in the literature shows that  the 
space order of convergence in \thmref{fullweak} is far to be optimal for $\beta=1$.
 \begin{theorem}
 \label{opfullaplace}
  Assume that $A=\varDelta$ with  Dirichlet boundary condition ($V=H_0^1(\Omega)$), assume that the noise is additive, \assref{ass:driftandB} is satisfied with
  $\Vert (-A)^{\frac{\beta-1}{2}}Q^{1/2}\Vert_{\mathcal{L}_{2}(H)}<\infty$ for some $\beta \in [1/2,1]$. 
  Furthermore assume that $X_0\in \mathcal{D}((-A)^{\beta})$.
 For $\Phi\in \mathcal{C}_2^b(H;\mathbb{R})$ and  arbitrary small $\epsilon>0,$ the following estimation hold
\begin{equation} 
|\EE[\Phi(X_M^h)] - \EE[\Phi(X(T)]|\leq C(\Delta t^{\beta-\epsilon}+h^{2\beta-\epsilon}),
\end{equation}
where the constant $C$ depends on $\alpha, \,\epsilon,\, ...,T,\,L$ and the initial data, but is independent of $h$ and $M.$ 
   \end{theorem}
 Let us prove \thmref{opfullaplace}.  
 
 \begin{proof}
  The proof follows the one for \thmref{fullweak} and we have
 \begin{eqnarray}
  |\EE[\Phi(X_M^h)] - \EE[\Phi(X(T)]| \leq|\EE[\Phi(X_M^h)] - \EE[\Phi(X^h(T)]|+|\EE[\Phi(X^h(T)] - \EE[\Phi(X(T)]|.
  \end{eqnarray}
  If  $X_0\in \mathcal{D}((-A)^{\beta}),\, 1/2  \leq \beta\leq 1$, then $X_0\in \mathcal{D}((-A)^{1/2})$ and  
   all the lemma used in \thmref{thm:mainadditive} and \thmref{thm:multiplicative} are still valid.
From \thmref{thm:mainadditive},  if $\Vert (-A)^{\frac{\beta-1}{2}}Q^{1/2}\Vert_{\mathcal{L}_{2}(H)}<\infty$ for some $\beta \in [1/2,1]$ we have
\begin{eqnarray}
  |\EE[\Phi(X_M^h)] - \EE[\Phi(X^h(T)]|\leq C\Delta t^{\beta-\epsilon}.
\end{eqnarray}
Note that  \assref{ass:driftandB} implies that $F \in \mathcal{C}_2^b(H;H)$.  For $\Vert (-A)^{\frac{\beta-1}{2}}Q^{1/2}\Vert_{\mathcal{L}_{2}(H)}<\infty$ for some $\beta \in [1/2,1]$, 
\cite[Assumption A, Theorem 1.1]{stignonlinearweak} gives 
\begin{eqnarray}
\label{spacedis}
 |\EE[\Phi(X^h(T)] - \EE[\Phi(X(T)]|\leq  h^{2\beta-\epsilon}.
\end{eqnarray}
The proof of \eqref{spacedis} in \cite{stignonlinearweak} uses some elements of Malliavin calculus.
\end{proof}
The following remark generalizes the \thmref{opfullaplace} for general selfadjoint operator with not necessarily Dirichet boundary condition.
   \begin{remark}
   \label{oppta}
For additive noise and under the same condition as in \thmref{opfullaplace}, if $A$ is selfadjoint, and $\Vert (-A)^{\frac{\beta-1}{2}}Q^{1/2}\Vert_{\mathcal{L}_{2}(H)}<\infty$ for some $\beta \in [k,1], \; k\geq 1/2$ 
and  $X_0\in \mathcal{D}((-A)^{\beta})$.
For $\Phi\in \mathcal{C}_2^b(H;\mathbb{R})$ and  arbitrary small $\epsilon>0,$ the following estimation hold
\begin{equation} 
\label{optt}
|\EE[\Phi(X_M^h)] - \EE[\Phi(X(T)]|\leq C(\Delta t^{\beta-\epsilon}+h^{2\beta-\epsilon}).
\end{equation}
The proof is the same as in \cite[Assumption A, Theorem 1.1]{stignonlinearweak} where  the Laplace operator is used just for simplicity. 
The difference comes from  the set of  the eigenvalues of the self adjoint operator  $A$. Once the range  of $\beta$ 
such that  $\Vert (-A)^{\frac{\beta-1}{2}}Q^{1/2}\Vert_{\mathcal{L}_{2}(H)}<\infty$ is found, the proof of  \cite[Assumption A, Theorem 1.1]{stignonlinearweak} is applied line by line.
\end{remark}
The following numerical simulations will confirm  numerically estimation \eqref{optt} of \rmref{oppta}.
\section{Numerical Simulations}
\label{sec6}
We consider the reaction diffusion equation 
\begin{eqnarray}
\label{linear}
 dX=(D \varDelta X -0.5 X)dt+ dW 
\qquad \text{given } \quad X(0)=X_{0},
\end{eqnarray}
on the time interval $[0,T]$ 
% {\bf ANTOINE : WAS $D=1$ ?  AND $L_1=L_2=1$ ?
%   yes it was!!  }
and homogeneous Neumann boundary conditions on the domain
$\Omega=[0,L_{1}]\times [0,L_{2}]$.   The  noise is represented by \eqref{eq:W}.
The eigenfunctions $\{ e_{i,j}\}_{i,j\geq 0}=\{e_{i}^{(1)}\otimes e_{j}^{(2)}\}_{i,j\geq 0}
$ of the operator $-\varDelta$ here are given by 
\begin{eqnarray*}
e_{0}^{(l)}(x)=\sqrt{\dfrac{1}{L_{l}}},\qquad 
e_{i}^{(l)}(x)=\sqrt{\dfrac{2}{L_{l}}}\cos(\lambda_{i}^{(l)}x), \qquad \lambda_{0}^{(l)}=0,\qquad
\lambda_{i}^{(l)}=\dfrac{i \pi }{L_{l}}
\end{eqnarray*}
where $l \in \left\lbrace 1, 2 \right\rbrace,\, x\in \Omega$ and  $i \in \mathbb{N}$
with the corresponding eigenvalues $ \{\lambda_{i,j}\}_{i,j\geq 0} $ given by 
$\lambda_{i,j}= (\lambda_{i}^{(1)})^{2}+ (\lambda_{j}^{(2)})^{2}$. We
take $L_1=L_2=1$. 
Notice that $A=D \varDelta $ does not satisfy \assref{assumptionn} as
$0$ is an eigenvalue. 
To eliminate the eigenvalue $0$ we use the perturbed operator $A=D \varDelta+\epsilon \mathbf{I},\, \epsilon>0$. 
%or by simply eliminating the node with eigenvalue $0$ if $q_{0}=0$.  
The exact solution of \eqref{linear} is known . Indeed 
the decomposition of \eqref{linear} in each eigenvector node yields the following  Ornstein-Uhlenbeck process
\begin{eqnarray}
\label{exact}
 dX_{i}=-(D \lambda_{i}+0.5)X_{i}dt+ \sqrt{q_{i}}d\beta_{i}(t)\qquad i \in \mathbb{N}^{2}.
\end{eqnarray}
This is a Gaussian process with the mild solution
\begin{eqnarray}
X_{i}(t)= e^{-k_{i}t}X_{i}(0)+  \sqrt{q_{i}}\int_{0}^{t}e^{k_{i}(s-t)} d \beta_{i}(s),\quad k_{i}=D \lambda_{i}+0.5,
\end{eqnarray}
which  is therefore an Ornstein-Uhlenbeck process.
Applying the It\^{o} isometry yields the following exact variance of $X_{i}(t)$
\begin{eqnarray}
 \text{Var}(X_{i}(t))=\dfrac{q_{i}}{2 k_{i}}\left(1-e^{-2 \,k_{i}\,t}\right).
\end{eqnarray}
During the simulations, we compute the exact solution recurrently as 
\begin{eqnarray}
\label{exact}
X_{i}^{m+1}&=& e^{-k_{i} \Delta t}X_{i}^m+  \sqrt{q_{i}}\int_{t_m}^{t_{m+1}}e^{k_{i}(s-t)} d \beta_{i}(s)\nonumber\\
            &=& e^{-k_{i} \Delta t }X_{i}^m +\left(\dfrac{q_{i}}{2 k_{i}}\left(1-e^{-2 \,k_{i}\,\Delta t}\right)\right)^{1/2}R_{i,m},
\end{eqnarray}
where $R_{i,m}$ are independent, standard normally distributed random
variables with mean $0$ and variance $1$.
 The expression in \eqref{exact}  allows to use the same set of random numbers for both the exact and the numerical solutions. 

Our function $F(u)=-0.5 u $ is linear and obviously satisfies the Lipschitz condition in \assref{assumption1}. 
We assume that the covariance operator $Q$ and $A$ have the same eigenfunctions and we take the eigenvalues of the covariance operator to be
\begin{eqnarray}
\label{noise2}
 q_{i,j}=\left( i^{2}+j^{2}\right)^{-(\beta+\delta)}, \beta>0
\end{eqnarray} 
in the representation \eqref{eq:W} for some small $\delta>0$.
Indeed for $ \beta \in [0,1]$  we have 
\begin{eqnarray*}
\Vert (-A)^{\frac{\beta-1}{2}}Q^{1/2}\Vert_{\mathcal{L}_{2}(H)}<\infty \Leftrightarrow
\underset{(i,j) \in \mathbb{N}^{2}}{\sum}\lambda_{i,j}^{\beta-1}q_{i,j}<  \pi^{2}\underset{(i,j) 
\in \mathbb{N}^{2}}{\sum} \left( i^{2}+j^{2}\right)^{-(1+\delta)} <\infty.
\end{eqnarray*}
We will consider the following two test functions
\begin{eqnarray}
 \Phi_1: f  \rightarrow \int_\Omega f(x) dx, \qquad \qquad \qquad  \Phi_2: f  \rightarrow \Vert f\Vert^2.
\end{eqnarray}
which obviously belong to $\mathcal{C}_{b}^{2}(H,\mathbb{R})$. 
 By setting $B=D \varDelta -0.5 I$, using the fact  the the Ito's integral vanishes, for $X_0 \in H$ we have
\begin{eqnarray}
\label{exact1}
 \EE\Phi_1(X(t))&=& \int_{\Omega} e^{t B}X_0(x) dx+ \EE \left[ \int_\Omega \left(\int_0^t e^{(t-s) B} dW(s,x)\right)dx \right ]\nonumber \\
     &=& \int_{\Omega} e^{t B}X_0(x) dx+  \int_\Omega \EE \left[\int_0^t e^{(t-s) B} dW(s,x) \right ]dx\nonumber \\
     &=& \int_{\Omega} e^{t B}X_0(x) dx \nonumber\\
     &=&\underset{i,j \geq 0}{\sum^{\infty}}e^{-k_{i j} t} \langle e_{i,j},X_0\rangle_H\int_{\Omega} e_{i j}(x)dx, \qquad k_{i j}=D\lambda_{i,j}+0.5.
  \end{eqnarray}
We also have 
\begin{eqnarray}
\label{exact22}
 \EE\Phi_2(X(t))&=&\EE\Vert e^{t B}X_0+\int_0^t e^{(t-s) B} dW(s)\Vert^{2}\\
                 &=&  \EE \Vert e^{t B}X_0\Vert^2+    \EE \Vert \int_0^t e^{(t-s) B} dW(s) \Vert^2+    2\EE \langle e^{t B}X_0, \int_0^t e^{(t-s) B} dW(s) \rangle_H
 \end{eqnarray}
 Using the spectral decomposition, the fact that $X_0 \in H$ and the Ito isometry yields
 \begin{eqnarray}
 \label{exact2}
  \EE\Phi_2(X(t))=\underset{i,j \geq 0}{\sum^{\infty}} e^{-2k_ij} \langle e_{i,j},X_0\rangle_{H}^2+ \underset{i,j \geq 0}{\sum^{\infty}} \langle \dfrac{q_{i j}}{2 k_{i j}}\left(1-e^{-2 \,k_{i j}\,\Delta t}\right)\rangle_H.
 \end{eqnarray}
 To compute  the exact solutions $\EE\Phi_1(X(t))$ and $\EE\Phi_2(X(t))$, we can either truncate \eqref{exact1} and \eqref{exact2} by using the first $N_h$ terms, $N_h$ being the number of finite element test basis functions,
 or use directly the Monte Carlo method with \eqref{exact}. 
 
 Our code was implemented in Matlab 8.1. We use two different intial solutions with each test function $\Phi_1$ and $\Phi_2$. Indeed we use  the initial solution $X_0=X_0^{(2)}=0$ for $\Phi_2$, and  
 $$X_0=X_0^{(1)}= \underset{i,j \geq 1}{\sum^{\infty}} q^1_{i,j} e_{i,j},\;\;\;q^1_{i,j}=(i^2+j^2)^{-1.001}$$ for for $\Phi_1$. 
 Our finite element triangulation has been done from rectangular grid of maximal length $h$.  Fast Leja Points technique as presented in \cite{TLG,SebaGatam} 
 is used to compute the exponential matrix function $S_h^1= \varphi_1$.
 In the legends of our graphs, we use the following notations
\begin{itemize}
 \item  ''TimeError1'' denotes the weak error for fixed $h=1/150$  with $\Phi=\Phi_1$, $X_0=X_0^{(1)}$ and final time $T=1$.
 \item  ''TimeError2'' denotes the weak error for fixed $h=1/50$ with $\Phi=\Phi_2$, $X_0=X_0^{(2)}$ and final time $T=1$.
 \item  ''SpaceError1'' denotes the weak error for fixed time step $\Dt=1/500$  with $\Phi=\Phi_1$,  $X_0=X_0^{(1)}$ and final time $T=0.1$.
 \item  ''SpaceError2'' denotes the weak error for fixed time step  $\Dt=1/20000$ with $\Phi=\Phi_2$ and  $X_0=X_0^{(2)}$ and final time $T=0.1$.
\end{itemize}
In all our graphs $D=0.1, \beta=1$, $\delta=0.001$, the weak errors are computed at the final time $T$ and $50$ realizations are used to estimate the weak errors with the Monte Carlo method.
\begin{figure}[!ht]
  \subfigure[]{ 
    \label{FIG01a}
    \includegraphics[width=0.48\textwidth]{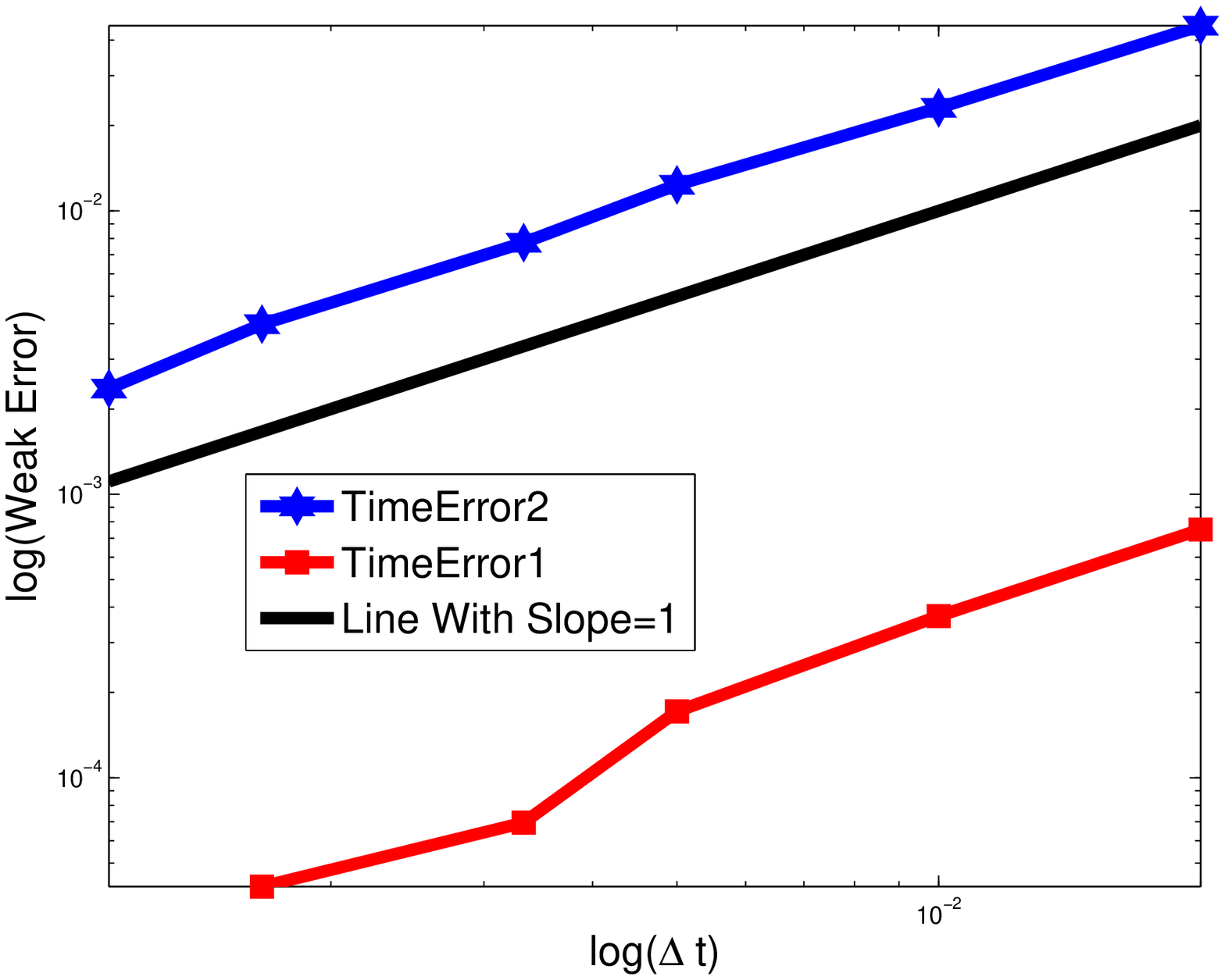}}
  \hskip 0.01\textwidth
  \subfigure[]{
    \label{FIG01b}
    \includegraphics[width=0.48\textwidth]{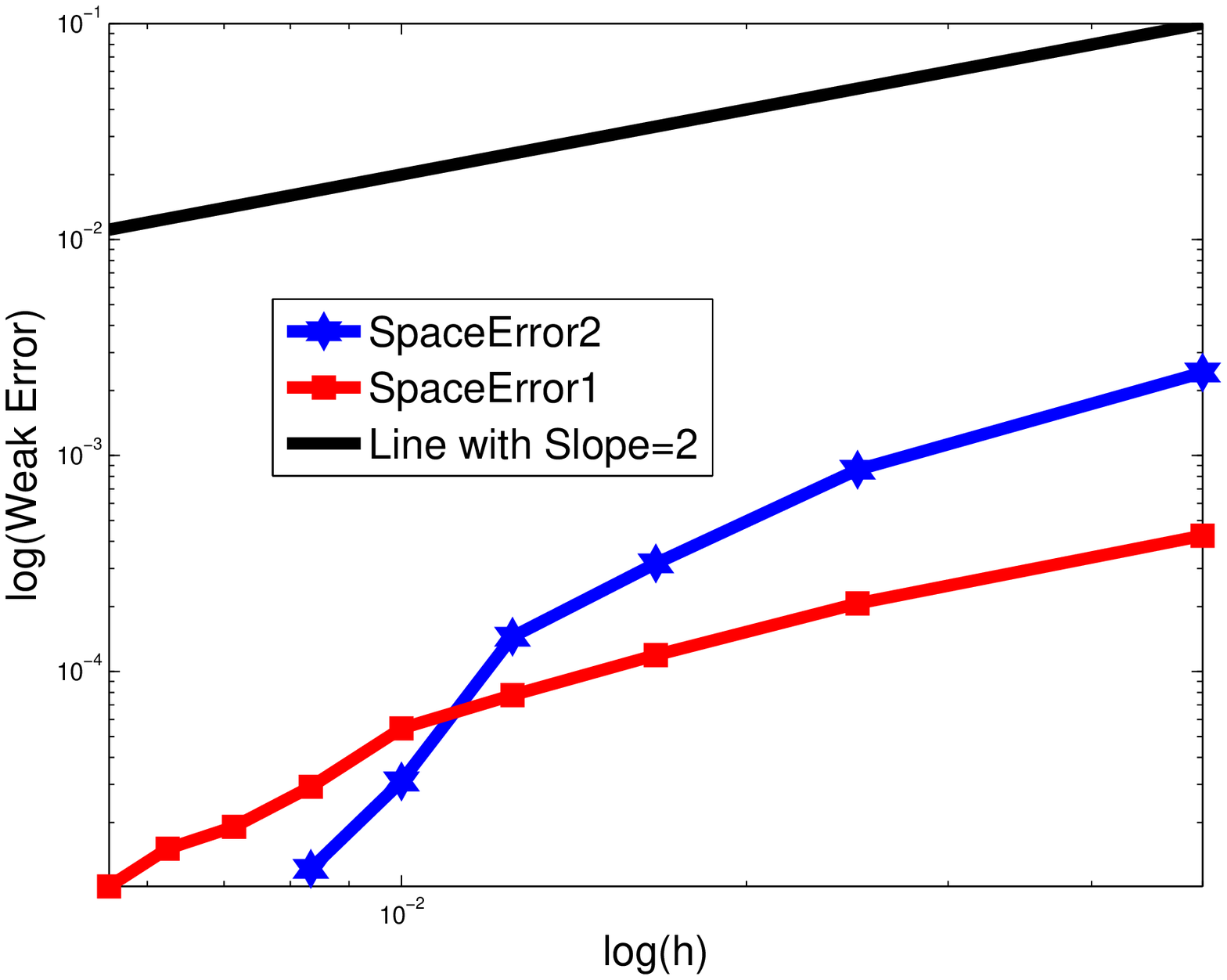}}
  \caption{(a) Weak convergence  in time of  the exponential scheme for $\Phi=\Phi_1$  with $X_0=X_0^{(1)}$ (TimeError1) and $\Phi=\Phi_2$  with $X_0=X_0^{(2)}$ (TimeError2). (b)
  Weak convergence  in space of  the exponential scheme for $\Phi=\Phi_1$  with $X_0=X_0^{(1)}$ (SpaceError1), and $\Phi=\Phi_2$  with $X_0=X_0^{(2)}$ (SpaceError2).
  }
  \label{FIG01}
\end{figure}

\figref{FIG01a} shows  the weak convergence  in time of  the exponential scheme for $\Phi=\Phi_1$  with $X_0=X_0^{(1)}$, and $\Phi=\Phi_2$  with $X_0=X_0^{(2)}$. 
The weak order of convergence in time is respectively  $1.1$ for $\Phi=\Phi_1$ and $1.008$ for $\Phi=\Phi_2$. 
These orders of convergence are then closed to the optimal order $1$ obtained in \rmref{oppta}.

For space convergence, very small time steps are needed and the weak errors are performed at small final  time $T=0.1$, \figref{FIG01b} shows  the weak convergence  in space of  the exponential scheme for $\Phi=\Phi_1$  with $X_0=X_0^{(1)}$, and $\Phi=\Phi_2$  with $X_0=X_0^{(2)}$. 
The weak order of  convergence in space is respectively  $1.71$ for $\Phi=\Phi_1$ and $2.01$ for $\Phi=\Phi_2$. 
These orders of convergence are also closed to the optimal order $2$ obtained in \rmref{oppta}. To sum  up, our simulations confirm the theoretical results obtained in \rmref{oppta}.

%\begin{acknowledgments}
\section*{Acknowledgements}
This project was supported by the Robert Bosch
Stiftung through the AIMS ARETE chair programme.
%\end{acknowledgments}

% 
% 
% STATE CONVERGENCE RESULT (SAME AS FOR ADDITIVE NOISE) AND INDICATE THE PART OF THE PROOF THAT NEED SOME CHANGES. PERFORM THE CHANGES AND CONCLUDE. COMMENT ON THE USE OF THE LINEAR GROWTH CONDITION AND COMPARE WITH THE ADDITIVE NOISE CONDITION.
% 
% 
% TO BE DONE:
% -GENERAL EDITING OF THE PAPER WITH MORE DETAILED PROOFS. 
% CHECK THAT THE LABEL AND REF ARE DONE APPROPRIATELY
% -STRONG CONVERGENCE RATE ANALYSIS COMPARISON WITH WEAK CONVERGENCE(ANTOINE)
% -NUMERICAL SIMULATION FOR SIMPLE EXAMPLE AND PRACTICAL EXAMPLE AND RELATIONSHIP WITH THE THEORY(ANTOINE)
% -WRITE A WELL FORMULATED INTRODUTION WITH COMPLETE  LITTERATURE REVIEW (ANTOINE)
% -COMPLETE THE LIST OF REFERENCES 
% -PROOF READING OF THE ENTIRE PAPER
  
%  \begin{thebibliography}{99}
% %\bibliographystyle{abbrv}
% %1
% \harvarditem{Da~ Prato \& Zabczyk}{1992}{DaPZ}
% \textsc{ Da~Prato,G.\& Zabczyk, J.} (1992) Stochastic Equations in Infinite Dimensions,  volume~ \textbf{44} of {\em
%  Encyclopedia of Mathematics and its Applications, Cambridge University Press, Cambridge, Uk}.
%  %2
% \harvarditem{Pr\'ev\^ot \& R\"ockner}{2007}{PrvtRcknr}
% \textsc{ Pr\'ev\^ot, C.\& R\"ockner, M.} (2007)
% A Concise Course on Stochastic Partial Differential Equations, { \em Springer}, ISBN-10: 3540707808.
% %3
% \harvarditem{Chow, P-L}{2007}{Chw}
% \textsc{Chow, P-L} (2007) Stochastic Partial Differential Equations,
%  {\em Applied Mathematics and nonlinear Science. Chapman \& Hall/ CRC},
%  ISBN-1-58488-443-6.     
% \end{thebibliography}
%\end{document}
%           
% 
\section*{REFERENCES}
\bibliographystyle{abbrv}
 %\bibliography{weakarticle.bib}
 \bibliography{weakconvergenceSPDE_July6.bbl}
\end{document}